\documentclass[12pt,reqno]{amsart}
\usepackage{amssymb, a4wide, amsmath, mathtools,xy,amsthm}
\xyoption{all}
\usepackage{latexsym,pdfsync,xcolor,graphicx,booktabs}
\usepackage{mathtools,tikz-cd}
\usepackage{multirow}
\usepackage{pbox}
\usepackage{orcidlink}

\usepackage{longtable}

\allowdisplaybreaks
\usepackage{array}
\usepackage{upgreek}
\usepackage{pstricks-add}
\usepackage{enumerate}

\usepackage[english]{babel}	
\usepackage{comment}
\allowdisplaybreaks

\usepackage{pgf,tikz}
\usepackage{wrapfig}
\usetikzlibrary{arrows}

\usepackage{float}
\usepackage{appendix}

\usepackage{hyperref}
\usepackage{rotating}
\usepackage{etoolbox}

\makeatletter
\def\l@section{\@tocline{1}{10pt}{0em}{2.5em}{\bfseries}}
\def\l@subsection{\@tocline{2}{0pt}{2.5em}{3.5em}{}}
\def\l@subsubsection{\@tocline{3}{0pt}{5em}{5.5em}{}}
\makeatother

\AtBeginEnvironment{subappendices}{%
	\section{Appendices}
	\addcontentsline{toc}{section}{\textbf{Appendices}}
	\counterwithin{figure}{section}
	\counterwithin{table}{section}
}

\AtEndEnvironment{subappendices}{%
	\counterwithout{figure}{section}
	\counterwithout{table}{section}
}

\newtheorem{theorem}{Theorem}[section]

\newtheorem{corollary}[theorem]{Corollary}
\newtheorem{proposition}[theorem]{Proposition}
\newtheorem{notation}[theorem]{Notation}
\newtheorem{lemma}[theorem]{Lemma}

\theoremstyle{definition}
\newtheorem{definition}[theorem]{Definition}
\newtheorem{example}[theorem]{Example}

\theoremstyle{remark}
\newtheorem{remark}[theorem]{Remark}

\newcommand{\E}{\mathcal{E}}

\DeclareMathOperator{\End}{End}

\DeclareMathOperator{\ann}{ann}
\DeclareMathOperator{\spa}{\textup{\textsf{span}}}

\numberwithin{equation}{section}
\title[On formal deformations and degenerations of evolution algebras]{On formal deformations and degenerations of evolution algebras}

\author[A. Makhlouf]{Abdenacer Makhlouf\textsuperscript{1}\,\orcidlink{0000-0002-5329-487X}}
\address{\textsuperscript{1} Université de Haute Alsace, IRIMAS, Département de Mathématiques,  F-68093 Mulhouse, France}
\email{abdenacer.makhlouf@uha.fr}

\author[A. Pérez-Rodríguez]{Andrés Pérez-Rodríguez\textsuperscript{2}\,\orcidlink{0009-0007-1095-5328}}
\address{\textsuperscript{2} Universidade de Santiago de Compostela, Departmento de Matemáticas \& CITMAga,  15782 Santiago de Compostela, Spain}
\email{andresperez.rodriguez@usc.es}



\begin{document}
	
	\begin{abstract}
		The main purpose of this paper is to study formal deformations of evolution algebras, determining their existence and classifying them up to equivalence. In addition, we examine degenerations in this setting and provide Hasse diagrams that describe the degeneration relations among nilpotent evolution algebras of dimensions up to four.
	\end{abstract}
	\maketitle
	\vspace{-0.5cm}
	\section{Introduction}
	Evolution algebras are commutative but non-associative structures introduced by Tian and Vojt\v{e}chovsk\'{y} in 2006 (see \cite{TV_06}), motivated by applications in non-Mendelian genetics and dynamical systems. Two years later, Tian published a new monograph (see \cite{Tian_08}) in which the properties and biological applications of evolution algebras are studied in more detail. Unlike many classical algebraic structures, evolution algebras are not defined by polynomial identities but rather by the existence of a distinguished basis in which the product of any two distinct basis elements is zero. Hence, evolution algebras, in general, do not belong to any of the well-known classes of non-associative algebras such as Lie, Jordan, or alternative algebras. 
	This distinction is particularly relevant to the main goal of this work: the study of \textit{deformations} and \textit{degenerations} of evolution algebras, which have received very little attention in this setting, despite having been extensively investigated in the aforementioned classical varieties. To the best of our knowledge, the only related study in this direction is \cite{CFK_24}, which deals with degenerations of evolution algebras with one-dimensional squares.
	
	Although there are numerous definitions and perspectives for both concepts, we adopt here a more formal point of view.
	\textit{Formal deformations}, introduced by Gerstenhaber \cite{G_64} for associative algebras and later generalized to Lie algebras by Nijenhuis and Richardson \cite{NR_66,NR_67}, provide a way to study how the multiplication of an algebra can be perturbed while remaining inside the same class. Roughly speaking, a deformation of an algebraic structure $\mathcal{A}$ with product $\mu$ consists in constructing a new product $\mu_t$ over the formal power series space $\mathcal{A}[[t]]$ given by $\mu_t=\mu+\sum_{k\geq1}t^k\mu_k$, where each $\mu_k$ is a bilinear map on $\mathcal{A}$. In general, the goal of deformation theory is to understand how new multiplications $\mu_t$ enrich or modify the original structure, by determining their existence and classifying them up to the so-called \textit{equivalence}, a process typically controlled by the second cohomology space.
	
	On the other hand, \textit{degeneration} is a concept that is somewhat opposite of deformation, and it has been extensively studied in the classical case of Lie algebras (see, for example, references \cite{BS_99_deg_Lie,Carl_78,S_90_deg_nilp_Lie}). Indeed, since Lie algebras are defined by polynomial identities, the set of $n$-dimensional Lie algebras over a field $\mathbb{K}$, $\mathcal{L}_n(\mathbb{K})$, forms an affine algebraic variety in the $n^3$-dimensional affine space $\mathbb{K}^{n^3}$. Moreover, the general linear group $\operatorname{GL}(n,\mathbb{K})$ acts naturally on $\mathcal{L}_n(\mathbb{K})$ by change of basis (see \cite{KN_87}). In this framework, given two $n$-dimensional Lie algebras $\mathcal{L}_1$ and $\mathcal{L}_0$ over a field $\mathbb{K}$, we say that $\mathcal{L}_1$ degenerates to $\mathcal{L}_0$, and write $\mathcal{L}_1\longrightarrow_{\text{deg}}\mathcal{L}_0$, if $\mathcal{L}_0$ lies in the Zariski closure of the orbit of $\mathcal{L}_1$ under this action. By the same arguments, a full description of degenerations for low-dimensional complex associative algebras has also been obtained (see, for example, references \cite{Gab_74,IP_24_deg_assoc,M_97,Maz_79}).
	
	However, since evolution algebras over a field $\mathbb{K}$ are not defined by identities, the $n$-dimensional ones do not form an affine algebraic variety of $\mathbb{K}^{n^3}$ and, consequently, the Zariski topology cannot be considered. Hence, we will adopt the formal viewpoint presented in \cite[Subsection~5.2]{M_07} for associative algebras. Given a continuous family $\{g_t\}_{t\neq0}$ of invertible linear maps on an $n$-dimensional vector space $V$ over $\mathbb{K}$
	and an algebra $\mathcal{A}_1$ over $\mathbb{K}$ with underlying vector space $V$ and product $\mu_1$, when the limit
	\begin{align}\label{eq:action}
		\mu_0(x,y)= \lim_{t\to0}g_t\cdot\mu_1(x,y)
		\coloneqq\lim_{t\to0}g_t\big(\mu_1(g_t^{-1}x,g_t^{-1}y)\big)
	\end{align}
	exists for all $x,y\in V$, we say that the algebra $\mathcal{A}_0$, with the same underlying vector space $V$ and product $\mu_0$, is a \textit{formal degeneration} of $\mathcal{A}_1$. As stated in \cite[Proposition~5.1]{M_07}, when working with algebraic varieties like associative or Lie algebras, every formal degeneration is also a degeneration in the usual sense. 
	
	The above definitions highlight the dual nature of the two concepts: formal degenerations tend to simplify the algebraic structure, often producing algebras that are closer to the abelian case, while formal deformations typically generate more intricate multiplication laws. In this work, we study both notions in the setting of evolution algebras, adapting the classical approaches and making the necessary changes due to the lack of variety structure but the existence of a natural basis.
	
	This paper is organized into five sections. Following this introduction, Section~\ref{sec:2} presents preliminary material on evolution algebras. We primarily focus on aspects related to nilpotent evolution algebras, such as the construction of the upper annihilating series and the classification in dimensions three and four (see Table~\ref{table:clas_nilp}), which will play a key role in the final part of the paper.
	
	In Section~\ref{sec:3}, we introduce formal deformations of evolution algebras (see Definition~\ref{def:def}) by requiring that the product of distinct elements of the natural basis remains zero. This condition naturally leads to an evolution algebra structure over the power series ring $\mathbb{K}[[t]]$. We also define the notion of equivalence of deformations and prove that, if two deformations are equivalent, the difference of their first-order terms takes a derivation-like expression (see Theorem~\ref{th:dif_infinit_terms}). Finally, we show that every evolution algebra admits a non-trivial deformation (see Theorem~\ref{th:nonrigid}), in sharp contrast to the rigidity typically observed in semisimple algebras. 
	
	Since the second cohomology group traditionally governs infinitesimal deformations, Section~\ref{sec:4} introduces a definition of the second cohomology space for evolution algebras (see Definition~\ref{def:cohom}), which likewise controls such deformations. As an illustrative example, we compute the second cohomology space of all two-dimensional complex evolution algebras, thereby obtaining all their infinitesimal deformations up to equivalence (see Theorem~\ref{th:cohom_2}).
	
	Finally, Section~\ref{sec:5} is devoted to the study of formal degenerations, understood as a sort of dual procedure to formal deformations. In particular, we establish several criteria to determine whether a degeneration exists (see Proposition ~\ref{prop:dim}). On the other hand, we show that degenerations lack transitivity (see Example~\ref{ex:1} and Remark~\ref{rem:with_example}), which can be seen as their main limitation. The section concludes with Hasse diagrams illustrating the degeneration relation among three and four-dimensional evolution algebras (see Theorem~\ref{th:deg_3} and Proposition~\ref{prop:deg_4}).

	\section{Some background on evolution algebras}\label{sec:2}
	Consider a vector space $V$ over a field $\mathbb{K}$ and a bilinear map $\mu\colon V\times V\longrightarrow V$. An \textit{evolution $\mathbb{K}$-algebra} is a pair $\mathcal{E}=(V,\mu)$ which admits a distinguished basis $B=\{e_1,\dots,e_n,\dots\}$, called \textit{natural basis}, such that $\mu(e_i,e_j)=0$ for all $i\neq j$. For simplicity of notation, we will often refer to an evolution algebra $\mathcal{E}$ directly by its product $\mu$, and we will write $e_ie_j$ instead of $\mu(e_i,e_j)$ when there is no risk of confusion. In particular, in this paper we focus on finite-dimensional evolution algebras, which means that $B$ is a finite set. For a given natural basis $B=\{e_1,\dots,e_n\}$, the scalars $\omega_{ij}\in\mathbb{K}$ satisfying $\mu(e_i,e_i)=\sum_{j=1}^n\omega_{ij}e_j$ are called the \textit{structure constants} of $\mathcal{E}$ relative to $B$. The matrix $M_B(\mathcal{E})=(\omega_{ij})_{i,j=1}^n$ is said to be the \textit{structure matrix} of $\mathcal{E}$ relative to $B$. Although we will extensively work with evolution algebras over $\mathbb{C}$, most  results will be stated and proved over arbitrary fields.
	
	\begin{notation}
		Let $V$ be a vector space over a field $\mathbb{K}$ with a fixed basis $\{e_1,\dots,e_n\}$. We denote by $\mathcal{Z}^2(V)$ the set of all bilinear maps on $V$ that vanish on distinct basis elements, that is, the space of bilinear maps defining evolution algebra structures on $V$:
		\[
		\mathcal{Z}^2(V) \coloneqq \left\{ \theta \in \operatorname{Bil}_{\mathbb{K}}(V \times V, V) \mid \theta(e_i, e_j) = 0 \text{ for all } i \neq j \right\}.
		\]
		Following the terminology commonly used in cohomology theory, we will refer to $\mathcal{Z}^2(V)$ as the \textit{space of $2$-cocycles}. Note that $\mathcal{Z}^2(V)$ can be naturally identified with the space of $n \times n$ matrices over $\mathbb{K}$, where $n = \dim(V)$, by associating each $\theta \in \mathcal{Z}^2(V)$ with its structure matrix relative to the fixed basis, disregarding possible isomorphisms.
	\end{notation}

	Recall that the \textit{annihilator} of an evolution algebra $\mathcal{E}$ with natural basis $B=\{e_1,\dots,e_n\}$ is characterised by \cite[Proposition~1.5.3]{thesis_yolanda},
	\[\ann(\mathcal{E})\coloneqq\{u\in\mathcal{E}\colon u\mathcal{E}=0\}=\spa\{e_i\in B\colon e_i^2=0\}.\]
	Moreover, following \cite[Definition 3.3]{EL_16}, we can also define the chain of ideals $\ann^i(\mathcal{E})$, $i\geq1$, where $\ann^1(\mathcal{E})\coloneqq\ann(\mathcal{E})$ and $\ann^i(\mathcal{E})$ with $i\geq2$ is defined by $\ann^i(\mathcal{E})/\ann^{i-1}(\mathcal{E})\coloneqq\ann(\mathcal{E}/\ann^{i-1}(\mathcal{E}))$. Equivalently, $\ann^i(\mathcal{E})\coloneqq\spa\{e\in B\colon e^2\in\ann^{i-1}(\mathcal{E})\}$ for all $i\geq2$. The chain of ideals
	\[0\subseteq\ann^1(\mathcal{E})\subseteq\dots\subseteq\ann^r(\mathcal{E})\subseteq\dots\]
	is called the \textit{upper annihilating series} of $\E$. Notice that, as we are only considering finite-dimensional algebras, there exists an integer $r\geq1$ such that $\ann^r(\mathcal{E})=\ann^{r+1}(\mathcal{E})=\ann^{r+2}(\mathcal{E})=\cdots$, that is, the upper annihilating series stabilises for some $r\geq1$. 
	
	In this paper, we will work extensively with nilpotent evolution algebras. Given a (not necessarily evolution) algebra $\mathcal{E}$, we define the following sequence of subalgebras:
	\[\mathcal{E}^1=\mathcal{E},\qquad  \quad  \ \ \mathcal{E}^{k+1}=\sum_{i=1}^k\mathcal{E}^i\mathcal{E}^{k+1-i}.\]	
	An (evolution) algebra $\mathcal{E}$ is called \textit{nilpotent} if there exists $n\in\mathbb{N}$ such that $\mathcal{E}^{n}=0$. Recall that the structure matrix of a nilpotent evolution algebra can be assumed to be strictly (upper or lower) triangular by \cite[Theorem~2.7]{CLOR_14}. Consequently, we also have that $\E$ is nilpotent if and only if the upper annihilating series reaches $\E$. That is, there exists an integer $r\geq1$ such that $\ann^r(\E)=\E$. Moreover, as presented in \cite[Definition 3.4]{EL_16}, the \textit{type} of an evolution algebra $\mathcal{E}$ is defined as the sequence $[n_1,\dots,n_r]$ such that
	\[n_i=\dim{\big(\ann(\E/\ann^{i-1}(\E))\big)}=\dim{\big(\ann^i(\E)\big)}-\dim{\big(\ann^{i-1}(\E)\big)},\]
	for all $i=1,\dots,r$. In the  following, nilpotent evolution algebras of dimension $n$ over a field $\mathbb{K}$ will be denoted by $\mathcal{N}_n(\mathbb{K})$. In particular, Table \ref{table:clas_nilp} displays the classification of nilpotent evolution algebras up to dimension four over $\mathbb{C}$, $\mathcal{N}_n(\mathbb{C})$ with $n=1,2,3$ and $4$, originally established in \cite[Theorems 5.1, 5.2, 5.3 \& 6.1]{HA_15} and later refined in \cite[Theorem 5.1]{EL_16}, together with their type and the dimension of their square.

    \begin{small}
       \begin{table}[H]
		\begin{center}
			\setlength{\tabcolsep}{5pt}
			\renewcommand{\arraystretch}{1.1}
			\begin{tabular}{||l | l | l | l ||}
				\hline \textbf{$\mathcal{E}$} & \textbf{Product}&\textbf{Type of $\mathcal{E}$}  & \textbf{Dimension of $\mathcal{E}^2$} \\
				\hline\hline
				$\mu_{1,1}$ &$e_1^2=0$ & $[1]$ & $0$ \\
				\hline
				$\mu_{2,1}$&$e_1^2=e_2^2=0$ & $[2]$ & $0$ \\
				$\mu_{2,2}$&$e_1^2=e_2,\;e_2^2=0$ & $[1,1]$ & $1$ \\
				\hline
				$\mu_{3,1}$&$e_1^2=e_2^2=e_3^2=0$ & $[3]$ & $0$ \\
				$\mu_{3,2}$&$e_1^2=e_2,\;e_2^2=e_3^2=0$ & $[2,1]$ & $1$ \\
				$\mu_{3,3}$&$e_1^2=e_2^2=e_3,\;e_3^2=0$ & $[1,2]$ & $1$ \\
				$\mu_{3,4}$&$e_1^2=e_2,\;e_2^2=e_3,\;e_3^2=0$ & $[1,1,1]$ & $2$ \\
                \hline
				 $\mu_{4,1}$&$e_1^2=e_2^2=e_3^2=e_4^2=0$ & $[4]$ & $0$ \\
				$\mu_{4,2}$&$e_1^2=e_2,\;e_2^2=e_3^2=e_4^2=0$ & $[3,1]$ & $1$ \\
				$\mu_{4,3}$&$e_1^2=e_2^2=e_3,\;e_3^2=e_4^2=0$ & $[2,2]$ & $1$ \\
				$\mu_{4,4}$&$e_1^2=e_3,\;e_2^2=e_4,\;e_3^2=e_4^2=0$ & $[2,2]$ & $2$ \\
				$\mu_{4,5}$&$e_1^2=e_2,\;e_2^2=e_4,\;e_3^2=e_4^2=0$ & $[2,1,1]$ & $2$ \\
				$\mu_{4,6}$&$e_1^2=e_2^2=e_3^2=e_4,\;e_4^2=0$ & $[1,3]$ & $1$ \\
				$\mu_{4,7}$&$e_1^2=e_2,\;e_2^2=e_3^2=e_4,\;e_4^2=0$ & $[1,2,1]$ & $2$ \\
				$\mu_{4,8}$&$e_1^2=e_2+ie_3,\;e_2^2=e_3^2=e_4,\;e_4^2=0$ & $[1,2,1]$ & $2$ \\
				$\mu_{4,9}$&$e_1^2=e_2^2=e_3,\;e_3^2=e_4,\;e_4^2=0$ & $[1,1,2]$ & $2$ \\
				$\mu_{4,10}$&$e_1^2=e_3,\;e_2^2=e_3+e_4,\;e_3^2=e_4,\;e_4^2=0$ & $[1,1,2]$ & $2$ \\
				$\mu_{4,11}$&$e_1^2=e_2,\;e_2^2=e_3,\;e_3^2=e_4,\;e_4^2=0$ & $[1,1,1,1]$ & $3$ \\
				$\mu_{4,12}$&$e_1^2=e_2+e_3,\;e_2^2=e_3,\;e_3^2=e_4,\;e_4^2=0$ & $[1,1,1,1]$ & $3$ \\
				\hline
			\end{tabular}
		\end{center}
		\caption{Classification of nilpotent evolution algebras up to dimension four over $\mathbb{C}$.}
		\label{table:clas_nilp}
	\end{table}	 
    \end{small}

	To conclude this preliminary section, and in contrast with nilpotent evolution algebras, we briefly recall the notion of regular evolution algebras. An evolution algebra $\E$ is said to be regular if $\E = \E^2$, meaning that it is generated by its squares. Since we are only considering finite-dimensional evolution algebras, this condition is equivalent to the structure matrix being non-singular.

	\section{Formal deformations of evolution algebras}\label{sec:3}
	Let $\E=(V,\mu)$ be a finite dimensional  evolution algebra over a field $\mathbb{K}$, and let $\mathbb{K}[[t]]$ denote the formal power series ring in one variable $t$. We define a formal space  $V[[t]]\coloneqq V\otimes\mathbb{K}[[t]]$, which is the result of extending the coefficient domain of $\E$ from $\mathbb{K}$ to $\mathbb{K}[[t]]$. Observe that each element $u\in V[[t]]$ can thus be written as a power series $u=\sum_{k\geq0}u_kt^k$, where  $u_k\in V$. Moreover, note that any $\mathbb{K}$-bilinear map $\nu\colon V\times V\longrightarrow V$ (in particular, every element of $\mathcal{Z}^2(V)$) extends naturally to a $\mathbb{K}[[t]]$-bilinear map from $V[[t]]\times V[[t]]$ to $V[[t]]$. 
	
	In this setting, the study of deformations aims to define new multiplication laws on the space $V[[t]]$ that yield new $\mathbb{K}[[t]]$-evolution algebra structures. This leads to the following definition.
	
	\begin{definition}\label{def:def}
		Let $\mathcal{E}=(V,\mu)$ be an evolution algebra with a natural basis $B=\{e_1,\dots,e_n\}$. A \textit{formal evolution deformation} of $\E$ is given by a $\mathbb{K}[[t]]$-bilinear map $\nu_t\colon V[[t]]\times V[[t]]\longrightarrow V[[t]]$ of the form $\nu_t=\mu+\sum_{k\geq1}\nu_kt^k$, where each $\nu_k\colon V\times V\longrightarrow V$ is a $\mathbb{K}$-bilinear map (extended to be $\mathbb{K}[[t]]$-bilinear) and the following ``condition of naturalness'' is satisfied:
		\begin{align}\label{eq:def_nat}
			\nu_t(e_i,e_j)=0,\quad\text{for all }i\neq j.
		\end{align}
	\end{definition}
	
	\begin{remark}~
		\begin{enumerate}
			\item Setting $t = 0$ in the previous definition recovers the original algebra $\mathcal{E}$.
			\item By bilinearity, a formal evolution deformation of an evolution algebra $\E$ is uniquely determined by how it acts on the elements of the natural basis of $\E$. 
			\item Note that condition \eqref{eq:def_nat} holds if and only if $\nu_k\in\mathcal{Z}^2(V)$ for all $k \geq 0$.
			Consequently, each map $\nu_k$ can be seen as an evolution algebra on its own. If we denote by $\rho_{ij}^k$ the structure constants of $\nu_k$ with respect to $B$, then the deformation $\nu_t$ expands as follows:
			\begin{align*}
				\nu_t(e_i,e_i)&=\mu(e_i,e_i)+\sum_{k\geq1}t^k\nu_k(e_i,e_i)\\
				&=\sum_{j=1}^n \omega_{ij} e_j+\sum_{k\geq1}t^k\left(\sum_{j=1}^n \rho_{ij}^k e_j\right)=\sum_{j=1}^n\left(\omega_{ij}+\sum_{k\geq1}\rho_{ij}^kt^k\right)e_j.
			\end{align*}
			\item As a consequence of condition \eqref{eq:def_nat}, a formal evolution deformation $\nu_t$ of an evolution algebra $\mathcal{E} = (V, \mu)$ defines an evolution $\mathbb{K}[[t]]$-algebra structure on the vector space $V[[t]]$, with the same natural basis as the original. It is worth mentioning that evolution algebras over rings (in particular, integral domains) have already been considered in~\cite{CMM_23}.
			\item From now on, and for the sake of simplicity, we will use the term deformation to refer to a formal evolution deformation.
		\end{enumerate}
	\end{remark}
	
	Although addressed from a different perspective, \cite{CLR_11_chain} provides several examples of chains of evolution algebras that, somehow, align with the notion of formal evolution deformation. Next, we present one of these examples.
	\begin{example}[{\cite[Example 1]{CLR_11_chain}}] This example models a time-homogeneous Markov process described in \cite{K_38}. In fact, using the Taylor series expansion of the functions $\sin(t)$, $\cos(t)$, and $e^t$, we obtain the following structure matrix:\\
		\begin{align*}
			\omega_{ii}^{[t]}&=\frac{2}{3}e^{-\frac{3}{2}At}\cos(\alpha t)+\frac{1}{3}\\
			&=\frac{2}{3}\left(\sum_{n=0}^{\infty}\frac{(-1)^n(\frac{3}{2}A)^n}{n!}t^n\right)\left(\sum_{n=0}^{\infty}\frac{(-1)^n\alpha^{2n}}{(2n)!}t^{2n}\right)+\frac{1}{3},\quad i=1,2,3;\\		
			\omega_{12}^{[t]}&=\omega_{23}^{[t]}=\omega_{31}^{[t]}=e^{-\frac{3}{2}At}\left(\frac{1}{\sqrt{3}}\sin(\alpha t)-\frac{1}{3}\cos(\alpha t)\right)+\frac{1}{3}\\
			&=\left(\sum_{n=0}^{\infty}\frac{(-1)^n(\frac{3}{2}A)^n}{n!}t^n\right)\left(\frac{1}{\sqrt{3}}\sum_{n=0}^{\infty}\frac{(-1)^n\alpha^{2n+1}}{(2n+1)!}t^{2n+1}-\frac{1}{3}\sum_{n=0}^{\infty}\frac{(-1)^n\alpha^{2n}}{(2n)!}t^{2n}\right)+\frac{1}{3};\\
			\omega_{21}^{[t]}&=\omega_{32}^{[t]}=\omega_{13}^{[t]}=-e^{-\frac{3}{2}At}\left(\frac{1}{\sqrt{3}}\sin(\alpha t)+\frac{1}{3}\cos(\alpha t)\right)+\frac{1}{3}\\
			&=-\left(\sum_{n=0}^{\infty}\frac{(-1)^n(\frac{3}{2}A)^n}{n!}t^n\right)\left(\frac{1}{\sqrt{3}}\sum_{n=0}^{\infty}\frac{(-1)^n\alpha^{2n+1}}{(2n+1)!}t^{2n+1}+\frac{1}{3}\sum_{n=0}^{\infty}\frac{(-1)^n\alpha^{2n}}{(2n)!}t^{2n}\right)+\frac{1}{3};	
		\end{align*}
		with $A>0$ and $\alpha=\frac{\sqrt{3}}{2}A$. It is easy to see that each entry of this matrix belongs to $\mathbb{K}[[t]]$, thus yielding a structure that could be treated as a  deformation.
	\end{example}
	\begin{definition}
		Let $\nu_t=\mu+\sum_{k\geq1}\nu_kt^k$ be a deformation of an evolution algebra $\mu$. Each coefficient $\nu_k$ is called the \textit{coefficient of order $k$}. Particularly, the first-order coefficient $\nu_1$ is called the \textit{infinitesimal} of $\nu_t$. Moreover, if $\nu_t$ is a truncated deformation, that is, there exists an integer $m$ such that $\nu_t=\mu+\sum_{k=1}^m\nu_kt^k$ with $\nu_m\neq0$, then we say that the deformation $\nu_t$ is of \textit{order} $m$.
	\end{definition}

	Analogously to Definition \ref{def:def}, one can also consider evolution algebra structures on $V \otimes \mathbb{K}[t]/(t^{m+1})$, which correspond to considering deformations up to order $m$, over $\mathbb{K}[t]/(t^{m+1})$. The next definition introduces a particularly important example of this.
	\begin{definition}
		Let $\E=(V,\mu)$ be an evolution algebra. Deformations of $\E$ over the vector space $V\otimes\mathbb{K}[t]/(t^2)$ are called \textit{infinitesimal deformations} of $\E$. The set of all such deformations will be denoted by $\operatorname{InfDef}(\E)$.
	\end{definition}
	\subsection{Equivalence of deformations}
	The next issue is to determine when two deformations should be regarded essentially the same. This is addressed by introducing the notion of equivalence between deformations.
	\begin{definition}
		Let $\nu_t$ and $\lambda_t$ be two formal evolution deformations of an evolution algebra $\E=(V,\mu)$. We say that $\nu_t$ and $\lambda_t$ are \textit{equivalent} if there exists a $\mathbb{K}[[t]]$-linear map $\phi_t\colon V[[t]]\longrightarrow V[[t]]$ of the form
		$\phi_t=\operatorname{Id}+t\phi_1+t^2\phi_2+\dots$, 
		where each $\phi_k\colon V\longrightarrow V$ is $\mathbb{K}$-linear and such that 
		\begin{align}\label{eq:equiv_def_1}
			\phi_t\big(\nu_t(u,v)\big)=\lambda_t\big(\phi_t(u),\phi_t(v)\big)
		\end{align}
		for all $u,v\in\mathcal{E}$. Notice that given a natural basis $\{e_1,\dots,e_n\}$ of $\E$, since $\nu_t$ and $\lambda_t$ are bilinear and $\phi_t$ is linear, the condition \eqref{eq:equiv_def_1} is equivalent to requiring that
		\begin{align}\label{eq:equiv_def}
			\phi_t\big(\nu_t(e_i,e_j)\big)=\lambda_t\big(\phi_t(e_i),\phi_t(e_j)\big)
		\end{align}
		holds for all $i,j=1,\dots,n$. If $\nu_t$ and $\lambda_t$ are equivalent through a formal isomorphism $\phi_t$, we write $\nu_t\cong_{\phi_t}\lambda_t$.
	\end{definition}
	\begin{remark}
		Given a ring $R$, it is known that an element $r=\sum_{k\geq0}r_kt^k$ is invertible in the ring of formal power series $R[[t]]$ if and only if $r_0$ is invertible in $R$. Consequently, any linear map $\phi\in\operatorname{End}_{\mathbb{K}}(V,V)[[t]]$ of the form $\phi=\operatorname{Id}+t\phi_1+t^2\phi_2+\dots$, as the one considered in the previous definition, is invertible and yields a $\mathbb{K}[[t]]$-automorphism of $V[[t]]$.
	\end{remark}
	Let $\E=(V,\mu)$ be an evolution algebra with natural basis $B=\{e_1,\dots,e_n\}$ and structure matrix $M_B(\E)=(\omega_{ij})$. Then, given two formal evolution deformations $\nu_t$ and $\lambda_t$ of $\mu$, we can express their products as
	\begin{align}\label{eq:def_formal_deformations}
		\begin{split}
			\nu_t(e_i, e_i) &= \sum_{j=1}^n \left( \omega_{ij} + {\rho}_{ij}^{\,1} t + {\rho}_{ij}^{\,2} t^2 + \dots \right) e_j = \sum_{j=1}^n\left(\omega_{ij}+\sum_{k\geq1}\rho_{ij}^kt^k\right)e_j,\\
			\lambda_t(e_i, e_i) &= \sum_{j=1}^n \left( \omega_{ij} + \sigma_{ij}^{\,1} t + \sigma_{ij}^{\,2} t^2 + \dots \right) e_j = \sum_{j=1}^n\left(\omega_{ij}+\sum_{k\geq1}\sigma_{ij}^kt^k\right)e_j,
		\end{split}
	\end{align}
	for all $i=1,\dots,n$, where $\rho_{ij}^k$ and $\sigma_{ij}^k$ are scalars in $\mathbb{K}$ representing the structure constants of the coefficients of each order $k$.
	\begin{proposition}\label{prop:equiv_deform}
		Let $\nu_t$ and $\lambda_t$ be two deformations of an evolution algebra $\E=(V,\mu)$ over any field $\mathbb{K}$, with expansions as in \eqref{eq:def_formal_deformations}. If they are equivalent, then there exists a matrix $(\xi_{ij})_{i,j=1}^n$ such that  the following conditions are satisfied:
		\begin{align}
			&\xi_{ji}\omega_{jk}+\xi_{ij}\omega_{ik}=0,\text{ for all }1\leq i, j, k\leq n\text{ such that }i\neq j;\label{cond_1}\\	
			&\rho_{ik}^{\,1}+\sum_{p=1}^{n}\omega_{ip}\xi_{kp}=\sigma_{ik}^1+2\xi_{ii}\omega_{ik},\text{ for all }1\leq i,k\leq n.\label{cond_2}
		\end{align}
	\end{proposition}
	\begin{proof}
		Let $\phi_t=\operatorname{Id}+t\phi_1+t^2\phi_2+\dots$ be a formal isomorphism such that $\nu_t\cong_{\phi_t}\lambda_t$, and denote by $(\xi_{ij})_{i,j=1}^n$ the matrix of $\phi_1$ with respect to the basis $B$. All the following computations are made modulo $t^2$. 	In view of \eqref{eq:equiv_def}, when $i\neq j$, we have 
		\begin{align*}
			0 = \phi_t\big(\mu_t(e_i, e_j)\big) 
			&= \lambda_t\big(\phi_t(e_i), \phi_t(e_j)\big) &&\\
			&= \lambda_t\big(e_i + t\phi_1(e_i),\; e_j + t\phi_1(e_j)\big) &\text{mod }(t^2)&\\
			&= \lambda_t(e_i, e_j) 
			+ t\big(\lambda_t(\phi_1(e_i), e_j) + \lambda_t(\phi_1(e_j), e_i)\big)  &\text{mod } (t^2)& \\
			&= t\left(
			\lambda_t\left(\sum_{p=1}^n \xi_{pi} e_p,\; e_j\right) 
			+ \lambda_t\left(\sum_{p=1}^n \xi_{pj} e_p,\; e_i\right)
			\right)  &\text{mod } (t^2) &\\
			&= t\big(\xi_{ji}\lambda_t(e_j, e_j) + \xi_{ij}\lambda_t(e_i, e_i)\big)  &\text{mod } (t^2)& \\
			&= t \left(
			\xi_{ji} \sum_{k=1}^n \omega_{jk} e_k 
			+ \xi_{ij} \sum_{k=1}^n \omega_{ik} e_k
			\right) &\text{mod } (t^2)&\\
			&= t \sum_{k=1}^n \left( \xi_{ji} \omega_{jk} + \xi_{ij} \omega_{ik} \right) e_k  &\text{mod } (t^2)&.
		\end{align*}
		Therefore, we obtain $\xi_{ji}\omega_{jk}+\xi_{ij}\omega_{ik}=0$ for all $1\leq i\neq j\leq n$ and for all $1\leq k\leq n$, which correspond exactly to the conditions presented in \eqref{cond_1}. Again, in view of \eqref{eq:equiv_def}, when $i=j$, we have\\
		\begin{equation}\label{eq:prop_1}
			\begin{aligned}
				\hspace{-0.1cm}\phi_t\big(\nu_t(e_i,e_i)\big)
				&=\sum_{k=1}^n \left( \omega_{ik} + {\rho}_{ik}^{\,1} t \right)\phi_t( e_k)&\text{mod } (t^2)&\\
				&=\sum_{k=1}^n \left( \omega_{ik} + {\rho}_{ik}^{\,1} t \right)\left(e_k+t\phi_1(e_k)\right)&\text{mod } (t^2)&\\
				&=\sum_{k=1}^n\omega_{ij}e_k+t\left(\sum_{k=1}^n\rho_{ik}^{\,1}e_k+\sum_{k=1}^n\omega_{ik}\phi_1(e_k)\right)& \text{mod } (t^2)&\\
				&=\sum_{k=1}^n\omega_{ik}e_k+t\left(\sum_{k=1}^n\rho_{ik}^{\,1}e_k+\sum_{k=1}^n\omega_{ik}\sum_{p=1}^n\xi_{pk}e_p\right)&\text{mod } (t^2)&\\
				&=\sum_{k=1}^n\omega_{ik}e_k+t\sum_{k=1}^n\left(\rho_{ik}^{\,1}+\sum_{p=1}^{n}\omega_{ip}\xi_{kp}\right)e_k&\text{mod } (t^2)&;
			\end{aligned}
		\end{equation}
		and\\	
		\begin{equation}\label{eq:prop_2}
			\hspace{-0.2cm}\begin{aligned}
				\lambda_t\big(\phi_t(e_i),\phi_t(e_i)\big)
				=&\lambda_t\left(e_i+t\phi_1(e_i),e_i+t\phi_1(e_i)\right)&\text{mod } (t^2)&\\
				=&\lambda_t(e_i,e_i)+2t\lambda_t\big(e_i,\phi_1(e_i)\big)&\text{mod } (t^2)&\\
				=&\lambda_t(e_i,e_i)+2t\lambda_t\left(e_i,\sum_{p=1}^{n}\xi_{pi}e_p\right)&\text{mod } (t^2)&\\
				=&\sum_{k=1}^n \left( \omega_{ik} + \sigma_{ik}^{\,1} t \right) e_k +2t\left(\xi_{ii}\sum_{k=1}^n \left( \omega_{ik} + \sigma_{ik}^{\,1} t \right) e_k\right)&\text{mod } (t^2)&\\
				=&\sum_{k=1}^n\omega_{ik}e_k+t\sum_{k=1}^n\left(\sigma_{ik}^{\,1}+2\xi_{ii}\omega_{ik}\right)e_k&\text{mod } (t^2)&.
			\end{aligned}
		\end{equation}
		Therefore, collecting the coefficients of first order in \eqref{eq:prop_1} and \eqref{eq:prop_2}, we obtain $\rho_{ik}^{\,1}+\sum_{p=1}^{n}\omega_{ip}\xi_{kp}=\sigma_{ik}^{\,1}+2\xi_{ii}\omega_{ik}$ for all $1\leq i,k\leq n$, getting the desired conditions~\eqref{cond_2}.
	\end{proof}
	Recall that a derivation of $\E$ is a linear map $d\colon\E\longrightarrow\E$ such that $d(uv)=d(u)v+ud(v)$, for all $u,v\in\E$. In particular, in \cite[Section 3.2.6]{Tian_08}, it was proved that, if $\E$ is an evolution $\mathbb{K}$-algebra with natural basis $B=\{e_1,\dots,e_n\}$, then a linear map $d$ such that $d(e_i)=\sum_{k=1}^{n}\xi_{ki}e_k$ is a derivation if and only if it satisfies the following conditions:
	\begin{align}
		&\xi_{ji}\omega_{jk}+\xi_{ij}\omega_{ik}=0,\text{ for all }1\leq i, j, k\leq n\text{ such that }i\neq j;\label{cond_1_der}\\
		&\sum_{p=1}^{n}\omega_{ip}\xi_{kp}=2\xi_{ii}\omega_{ik},\text{ for all }1\leq i,k\leq n.\label{cond_2_der}
	\end{align}
	Notice that conditions \eqref{cond_1}–\eqref{cond_2} already mirror those in \eqref{cond_1_der}–\eqref{cond_2_der}. This resemblance suggests that derivations underlie the equivalence of deformations, as made precise in the following theorem.	
	\begin{theorem}\label{th:dif_infinit_terms}
		Let $\nu_t$ and $\lambda_t$ be two deformations of an evolution algebra $\E=(V,\mu)$ over any field $\mathbb{K}$. If they are equivalent, then there exists a linear morphism $\varphi\in\End_\mathbb{K}(V)$ such that, for all $u,v\in V$,
		\begin{align}\label{eq_identity}
			\lambda_1(u,v)-\nu_1(u,v)=\varphi(uv)-u\varphi(v)-\varphi(u)v.
		\end{align}
	\end{theorem}
	\begin{proof}
		Given a formal isomorphism $\phi_t = \mathrm{Id} + t\phi_1 + t^2\phi_2 + \dots$ such that $\nu_t\cong_{\phi_t}\lambda_t$, let us set $\varphi \coloneqq \phi_1$, and denote its matrix with respect to the natural basis $B$ by $(\xi_{ij})_{i,j=1}^n$. We will verify \eqref{eq_identity} on the elements of the natural basis of $\mathcal{E}$, that is,
		\begin{align}\label{eq:th}
			\lambda_1(e_i,e_j)-\nu_1(e_i,e_j)=\phi_1(e_ie_j)-e_i\phi_1(e_j)-\phi_1(e_i)e_j,
		\end{align}
		for all $1\leq i,j\leq n$. The result will then follow by the bilinearity of $\nu_1$ and $\lambda_1$ and the linearity of $\phi_1$.
		On the one hand, when $i\neq j$, we have that $\nu_1(e_i,e_j)=\lambda_1(e_i,e_j)=0$ and, by Proposition~\ref{prop:equiv_deform}, it holds
		\begin{align*}
			\phi_1(e_ie_j)-e_i\phi_1(e_j)-\phi_1(e_i)e_j&=-e_i\left(\sum_{p=1}^n\xi_{pj}e_p\right)-e_j\left(\sum_{p=1}^n\xi_{pi}e_p\right)\\
			&=-\xi_{ij}e_i^2-\xi_{ji}e_j^2=-\sum_{k=1}^{n}\left(\xi_{ij}\omega_{ik}+\xi_{ji}\omega_{jk}\right)e_k=0.
		\end{align*}
		Consequently, \eqref{eq:th} is satisfied in this first case. On the other hand, when $i=j$, also as a consequence of Proposition \ref{prop:equiv_deform}, we have that
		\begin{align*}
			\phi_1(e_i^2)-2e_i\phi_1(e_i)&=\phi_1\left(\sum_{k=1}^{n}\omega_{ik}e_k\right)-2e_i\left(\sum_{p=1}^n\xi_{pi}e_p\right)=\sum_{k=1}^{n}\omega_{ik}\sum_{p=1}^n\xi_{pk}e_p-2\xi_{ii}e_i^2\\
			&=\sum_{k=1}^{n}\omega_{ik}\sum_{p=1}^n\xi_{pk}e_p-2\xi_{ii}\sum_{k=1}^n\omega_{ik}e_k=\sum_{k=1}^{n}\left(\sum_{p=1}^{n}\omega_{ip}\xi_{kp}-2\xi_{ii}\omega_{ik}\right)e_k\\
			&=\sum_{k=1}^{n}(\sigma_{ik}^1-\rho_{ik}^1)e_k=\lambda_1(e_i,e_i)-\nu_1(e_i,e_i),
		\end{align*}
		what completes the proof of \eqref{eq:th}.   
	\end{proof}
	
	\begin{definition}
		A deformation $\nu_t$ of an evolution algebra $\E=(V,\mu)$ is called \textit{trivial} if $\nu_t$ is equivalent to $\mu$.
	\end{definition}
	Hence, as a consequence of Theorem \ref{th:dif_infinit_terms}, we get the following result.
	\begin{corollary}\label{cor:def_non_trivial}
		Let $\nu_t$ be a deformation of an evolution algebra $\mathcal{E}=(V,\mu)$ over any field $\mathbb{K}$. If $\nu_t$ is a trivial deformation, then there exists a linear morphism $\varphi\in\End(V)$ such that $\nu_1(u,v)=\varphi(uv)-u\varphi(v)-\varphi(u)v$ for all $u,v\in V$.
	\end{corollary}

    \begin{remark}
    Note that all the above statements remain valid in characteristic two. 
    The only difference is that the terms of the form $2\xi_{ii}\omega_{ik}$
    appearing in \eqref{cond_2} and \eqref{cond_2_der} vanish automatically, so the identities are accordingly simplified.
\end{remark}
	
	\subsection{Every evolution algebra admits a non-trivial deformation}
	In many classical varieties of algebras, such as associative and Lie algebras, an algebra is called \textit{(formally) rigid} if all its formal deformations are trivial, a property often ensured by the vanishing of its second cohomology group. In contrast, we next show that  evolution algebras over any field always admit a nontrivial first-order deformation.
	
	\begin{lemma}\label{lem:lin_dep}
		Let $\nu_t$ and $\lambda_t$ be two equivalent deformations, $\nu_t\cong_{\phi_t}\lambda_t$, of an evolution algebra $\mathcal{E}=(V,\mu)$ with natural basis $\{e_1,\dots,e_n\}$. If $\sum_{i=1}^n\alpha_i\nu_t(e_i,e_i)=0$ for some scalars $\alpha_1,\dots,\alpha_n\in\mathbb{K}$, then $\sum_{i=1}^n\alpha_i\lambda_t(\phi_t(e_i),\phi_t(e_i))=0$.
	\end{lemma}
	\begin{proof}
		It follows straightforwardly from \eqref{eq:equiv_def}. In fact,
		\[0=\phi_t\left(\sum_{i=1}^n\alpha_i\nu_t(e_i,e_i)\right)=\sum_{i=1}^n\alpha_i\phi_t\big(\nu_t(e_i,e_i)\big)=\sum_{i=1}^n\alpha_i\lambda_t\big(\phi_t(e_i),\phi_t(e_i)\big).\]        
	\end{proof}
	\begin{theorem}\label{th:nonrigid}
		Every evolution algebra admits a nontrivial first-order deformation.
	\end{theorem}
	\begin{proof}
		Let $\mathcal{E}=(V,\mu)$ be an evolution algebra with natural basis $\{e_1,\dots,e_n\}$ and structure matrix $(\omega_{ij})_{i,j=1}^n$. We show that there always exists a nontrivial first-order deformation $\nu_t=\mu+\nu_1t$. For the computations that follow, denote by $(\rho_{ij}^1)_{i,j=1}^n$ the structure matrix of $\nu_1$. We consider the regular and the nonregular cases separately:
		
		\begin{enumerate}
			\item If $\mathcal{E}$ is regular then, from condition \eqref{cond_1}, it is deduced that 
			\begin{align}\label{eq:system}
				\begin{pmatrix}
					\omega_{i1} & \dots & \omega_{in}\\
					\omega_{j1} & \dots & \omega_{jn}
				\end{pmatrix}^t
				\begin{pmatrix}
					\xi_{ij}\\
					\xi_{ji}
				\end{pmatrix}=0
			\end{align}
			for all $1\leq i\neq j\leq n$. As $\mathcal{E}$ is regular, \eqref{eq:system} is a homogeneous system with a unique solution, and such solution is the trivial one. Then, $\xi_{ij}=0$ for all $1\leq i\neq j\leq n$. Next, from \eqref{cond_2} it is deduced that $\rho_{ii}^1=-\omega_{ii}\xi_{ii}$ for all $1\leq i\leq n$. If $\omega_{ii}=0$ for some $1\leq i\leq n$, then just consider a scalar $\rho_{ii}^1\neq0$, what yields that $\nu_t$ is a nontrivial deformation. Otherwise, if $\omega_{ii}\neq0$ for all $1\leq i\leq n$, we have that $\xi_{ii}=-\frac{\rho_{ii}^1}{\omega_{ii}}$ for all $1\leq i\leq n$. Consequently, also from \eqref{cond_2}, we get that necessarily $\rho_{ik}^1=\omega_{ik}(\xi_{kk}-2\xi_{ii})$. Then, taking a $\rho_{ik}^1$ which does not satisfy the previous relation yields again that $\nu_t$ is a nontrivial deformation, which completes the proof.
			\item If $\mathcal{E}$ is nonregular, consider the matrix $(\rho_{ij}^1)_{i,j=1}^n=\operatorname{diag}(\beta_1,\dots,\beta_n)$ with each $\beta_i=\pm1$. For the sake of contradiction, assume the existence of a formal isomorphism $\phi_t$ such that $\phi_t\big(\mu(e_i,e_i)\big)=\nu_t\big(\phi_t(e_i),\phi_t(e_i)\big)$ for all $1\leq i\leq n$ and for every $(\beta_1,\dots,\beta_n)\in\{\pm1\}^n$. In fact,
			\begin{align*}          \nu_t\big(\phi_t(e_i),\phi_t(e_i)\big)&=\nu_t(e_i,e_i)+2t\nu_t\big(e_i,\phi_1(e_i)\big)& \text{mod } (t^2)&\\
				&=\mu(e_i,e_i)+t\beta_ie_i+ 2t\xi_{ii}\big(\mu(e_i,e_i)+t\beta_ie_i\big) &\text{mod } (t^2)&\\
				&=\mu(e_i,e_i)+t\big(2\xi_{ii}\mu(e_i,e_i) +\beta_ie_i\big)& \text{mod } (t^2)&.
			\end{align*}
			At the same time, as $\mathcal{E}$ is not regular, there exists a non-empty subset $\Lambda\subset\{1,\dots,n\}$ such that $\sum_{i\in\Lambda}\alpha_i\mu(e_i,e_i)=0$ for some scalars $\alpha_i\in\mathbb{K}^*$, $i\in\Lambda$. Consequently, by Lemma~\ref{lem:lin_dep}, it also holds that $\sum_{i\in\Lambda}\alpha_i\nu_t\big(\phi_t(e_i),\phi_t(e_i)\big)=0$. Hence, looking at the coefficient of $t$, we have that
			\begin{align}\label{cond_3}
				0=\sum_{i\in\Lambda}\alpha_i\big(2\xi_{ii}\mu(e_i,e_i) +\beta_ie_i\big)\implies \sum_{i\in\Lambda}2\alpha_i\xi_{ii}\mu(e_i,e_i)= -\sum_{i\in\Lambda}\alpha_i\beta_ie_i.
			\end{align}
			We now distinguish two subcases according to the characteristic of the base field:
			\begin{enumerate}
				\item[2.1] If the characteristic is two, it suffices to take $\beta_i=1$ for all $i=1,\dots,n$. 
				In this case, condition \eqref{cond_3} yields $\sum_{i\in\Lambda}\alpha_ie_i=0$, a contradiction.
				\item[2.2] If the characteristic is not two, we claim  that \eqref{cond_3} is not satisfied for some $(\beta_1,\dots,\beta_n)\in\{\pm1\}^n$. Otherwise, if we changed the sign of $\beta_{i_0}$ for any index $i_0\in\varLambda$, there would also exist another formal isomorphism $\overline{\phi_t}$ and, consequently, other scalars $(\overline{\xi_{ii}})_{i\in\Lambda}$ such that
				\begin{align} \label{cond_4}\sum_{i\in\Lambda}2\alpha_i\overline{\xi_{ii}}\mu(e_i,e_i)=\beta_{i_0}\alpha_{i_0}e_{i_0}-\sum_{i\in\Lambda\backslash\{i_0\}}\alpha_i\beta_ie_i.
				\end{align} 
				By subtracting \eqref{cond_3} and \eqref{cond_4}, we get that
				\begin{align*}
					\sum_{i\in\Lambda}2\alpha_i(\xi_{ii}-\overline{\xi_{ii}})\mu(e_i,e_i)&=-2\beta_{i_0}\alpha_{i_0}e_{i_0},
				\end{align*}
				and consequently $e_{i_0}\in\spa\{\mu(e_i,e_i),i\in\Lambda\}$. Repeating this process for every $i_0\in\varLambda$, we get that $\spa\{e_i\colon i\in\Lambda\}=\spa\{\mu(e_i,e_i)\colon i\in\Lambda\}$, a contradiction with the fact that $\sum_{i\in\Lambda}\alpha_i\mu(e_i,e_i)=0$ for some $\alpha_i\in\mathbb{K}^*$, $i\in\Lambda$.
			\end{enumerate} 
			In both cases we get a contradiction, thus the result follows in the nonregular setting.
		\end{enumerate} 
		Therefore, it is always possible to construct a first-order nontrivial deformation, what yields the claim.
	\end{proof}
	Since all computations in the previous proof are done modulo $t^2$, the same conclusion applies in the context of infinitesimal deformations.
	\begin{corollary}\label{rem:inf_def}
		Every evolution algebra admits a nontrivial infinitesimal deformation.
	\end{corollary}
	\section{Formal deformations from a cohomological perspective}\label{sec:4}
	Note that Theorem \ref{th:dif_infinit_terms} expresses the difference of the infinitesimals of two equivalent deformations as a derivation-like expression. This result can be viewed as the natural analogue of a classical fact in deformation theory of associative and Lie algebras, where such difference is a $2$-coboundary in the Hochschild or Chevalley-Eilenberg cohomology, respectively (see \cite{G_64} for associative algebras and see \cite{NR_67} for Lie algebras). Moreover, in these two classical cases, the elements of the second cohomology group can be seen as infinitesimal deformations (up to equivalence).
	Although evolution algebras do not possess a standard cohomology theory, it is possible to construct a cohomological framework that captures, in an analogous way, the behaviour of infinitesimal deformations. 
	
	Let $\E=(V,\mu)$ be an evolution algebra. Inspired by the associative case, define the differential operator $\delta$ such that $\delta_\mu\varphi(u,v)\coloneqq\varphi(\mu(u,v))-\mu(u,\varphi(v))-\mu(\varphi(u),v)$ for all $\varphi\in\operatorname{End}_{\mathbb{K}}(V)$ and for all $u,v\in V$. Then, considering the image of $\delta_\mu$ restricted to the space of $2$-cocycles $\mathcal{Z}^2(V)$, we define the \textit{space of $2$-coboundaries}:
	\[\mathcal{B}^2(\mathcal{E})\coloneqq\{\theta\in\mathcal{Z}^2(V)\mid \theta=\delta_\mu\varphi\text{ for some }\varphi\in\operatorname{End}_{\mathbb{K}}(V)\}.\]
	It is straightforward to check that $\mathcal{B}^2(\E)$ is a   subspace of $\mathcal{Z}^2(V)$. Hence, we present the following definition.
	\begin{definition}\label{def:cohom}
		Let $\E=(V,\mu)$ be an evolution algebra. The \textit{second cohomology space} of $\E$ is defined as the quotient $\mathcal{H}^2(\E)\coloneqq\mathcal{Z}^2(V)/\mathcal{B}^2(\E)$.
	\end{definition}
	We can now state the following results in view of Theorem~\ref{th:dif_infinit_terms} and Corollary~\ref{cor:def_non_trivial}.
	\begin{corollary}
		If $\nu_t$ and $\lambda_t$ are two equivalent formal evolution deformations of an evolution algebra
		$\E$, then the difference of their infinitesimals is equal to zero  up to a $2$-coboundary.
	\end{corollary}
	\begin{corollary}\label{cor:inf_h2}
		Given an evolution algebra $\mathcal{E}$, there exists a bijection between its infinitesimal deformations (up to equivalence) and the elements of  $\mathcal{H}^2(\E)$.
	\end{corollary}
	Moreover, as a consequence of Theorem \ref{th:nonrigid} and Remark \ref{rem:inf_def}, we also have the following result.
	\begin{corollary}The cohomology space
		$\mathcal{H}^2(\E)$ is nontrivial for any finite-dimensional evolution algebra $\mathcal{E}$.
	\end{corollary}
	\subsection{Explicit computation of the second cohomology space}
	The computation of the second cohomology space (particularly, of the space of $2$-coboundaries) of an evolution algebra is feasible in low dimensions and becomes particularly accessible in cases with sparse structure matrices, such as the nilpotent setting. For the sake of completeness, we include here the explicit computation of this space for all two-dimensional evolution algebras over~$\mathbb{C}$, classified in \cite[Theorem~4.1]{CLOR_14}. As a consequence, we also obtain a characterisation (up to equivalence) of the infinitesimal deformations corresponding to each isomorphism class.
	
	\begin{theorem} \label{th:cohom_2}
		The spaces of $2$-coboundaries and all the infinitesimal deformations (up to equivalence) for all two-dimensional evolution algebras over $\mathbb{C}$ are presented in Table \ref{table:inf_def}.
		\begin{table}[H]
			\renewcommand{\arraystretch}{1.6}
			\resizebox{1\textwidth}{!}{
				\begin{tabular}{|| l | l  | l | l ||}
					\hline
					\textbf{$\E$} & \textbf{Product}  & \textbf{$\mathcal{B}^2(\E)$} & \textbf{$\operatorname{InfDef}(\E)$}  \\\hline\hline
					$\mathcal{E}_1$ & $e_1^2=e_1$, $e_2^2=0$   	& $\spa\left\{\big(\begin{smallmatrix}
						1 & 0 \\ 0 & 0
					\end{smallmatrix}\big),\big(\begin{smallmatrix}
						0 & 1 \\ 0 & 0
					\end{smallmatrix}\big)\right\}$ & $\big(\begin{smallmatrix}
						1 & 0 \\ 0 & 0
					\end{smallmatrix}\big)+\big(\begin{smallmatrix}
						0 & 0 \\ \alpha & \beta
					\end{smallmatrix}\big)t,\; \alpha,\beta\in\mathbb{C}$ \\
					\hline
					$\mathcal{E}_2$ & $e_1^2=e_2^2=e_1$ 	&  $\spa\left\{\big(\begin{smallmatrix}
						-1 & 0 \\ 1 & 0
					\end{smallmatrix}\big),\big(\begin{smallmatrix}
						0 & 1 \\ 0 & 1
					\end{smallmatrix}\big),\big(\begin{smallmatrix}
						0 & 0 \\ -2 & 0
					\end{smallmatrix}\big)\right\}$ & $\big(\begin{smallmatrix}
						1 & 0 \\ 1 & 0
					\end{smallmatrix}\big)+\big(\begin{smallmatrix}
						0 & 0 \\ 0 & \alpha
					\end{smallmatrix}\big)t,\; \alpha\in\mathbb{C}$  \\\hline
					$\E_3$ & $e_1^2=-e_2^2=e_1+e_2$  & $\spa\left\{\big(\begin{smallmatrix}
						1 & 1 \\ -1 & -1
					\end{smallmatrix}\big),\big(\begin{smallmatrix}
						0 & 1 \\ 2 & 1
					\end{smallmatrix}\big)\right\}$ & $\big(\begin{smallmatrix}
						1 & 1 \\ -1 & -1
					\end{smallmatrix}\big)+\big(\begin{smallmatrix}
						0 & 0 \\ \alpha & \beta
					\end{smallmatrix}\big)t,\; \alpha,\beta\in\mathbb{C}$ \\\hline
					$\E_4$ & $e_1^2=e_2$, $e_2^2=0$ &  $\spa\left\{\big(\begin{smallmatrix}
						0 & 1 \\ 0 & 0
					\end{smallmatrix}\big)\right\}$ & $\big(\begin{smallmatrix}
						0 & 1 \\ 0 & 0
					\end{smallmatrix}\big)+\big(\begin{smallmatrix}
						\alpha & 0 \\ \beta & \gamma
					\end{smallmatrix}\big)t,\; \alpha,\beta,\gamma\in\mathbb{C}$ \\\hline
					$\E_5(a_2,a_3)$ &$e_1^2=e_1+a_2e_2$  &$\spa\left\{\big(\begin{smallmatrix}
						-1 & -2a_2 \\ a_3 & 0
					\end{smallmatrix}\big),\big(\begin{smallmatrix}
						0 & a_2 \\ -2a_3 & -1
					\end{smallmatrix}\big)\right\}$& \text{\textbf{Case 1:}} $a_2\neq0$\\
					\hspace{0.5cm} &$e_2^2=a_3e_1+e_2$ & &  $\big(\begin{smallmatrix}
						1 & a_2 \\ a_3 & 1
					\end{smallmatrix}\big)+\big(\begin{smallmatrix}
						0 & 0 \\ \alpha & \beta
					\end{smallmatrix}\big)t,\; \alpha,\beta\in\mathbb{C}$ \\
					\hspace{0.5cm}& & &  \text{\textbf{Case 2:}} $a_2=a_3=0$ \\
					& & &   $\big(\begin{smallmatrix}
						1 & 0 \\ 0 & 1
					\end{smallmatrix}\big)+\big(\begin{smallmatrix}
						0 & \alpha \\ \beta & 0
					\end{smallmatrix}\big)t,\; \alpha,\beta\in\mathbb{C}$\\
					& && \text{\textbf{Case 3:}} $a_2=0,\;a_3\neq0$ \\
					& & &  $\big(\begin{smallmatrix}
						1 & 0 \\ a_3 & 1
					\end{smallmatrix}\big)+\big(\begin{smallmatrix}
						0 & \alpha \\ 0 & \beta
					\end{smallmatrix}\big)t,\; \alpha,\beta\in\mathbb{C}$\\\hline
					$\E_6(a_4)$ & $e_1^2=e_2$ & $\spa\left\{\big(\begin{smallmatrix}
						0 & 1 \\ -2 & -a_4
					\end{smallmatrix}\big),\big(\begin{smallmatrix}
						0 & -2 \\ 1 & 0
					\end{smallmatrix}\big)\right\}$& $\big(\begin{smallmatrix}
						0 & 1 \\ 1 & a_4
					\end{smallmatrix}\big)+\big(\begin{smallmatrix}
						\alpha & 0 \\ 0 & \beta
					\end{smallmatrix}\big)t,\; \alpha,\beta\in\mathbb{C}$\\
					&$e_2^2=e_1+a_4e_2$&&\\
					\hline
			\end{tabular}}
			\caption{Infinitesimal deformations and spaces $\mathcal{B}^2(\mathcal{E})$ for all 2-dimensional evolution algebras over $\mathbb{C}$.}
			\label{table:inf_def}
		\end{table}
	\end{theorem}
	\begin{proof}
		We prove the result in detail for the first evolution algebra in Table~\ref{table:inf_def}. 
		The remaining cases are analogous and can be found in Appendix~\ref{appen:1}.
		
		Let $\E_1$ be the evolution algebra with natural basis $B=\{e_1,e_2\}$ and product given by $e_1^2=e_1$ and $e_2^2=0$. If $\theta\in\mathcal{B}^2(\E_1)$, then there must exist a linear morphism $(\varphi)_{BB}=(\xi_{ij})$ such that the following conditions hold:
		\begin{align*}
			\theta(e_1,e_2)&=\varphi(e_1e_2)-e_1\varphi(e_2)-\varphi(e_1)e_2=-e_1(\xi_{12}e_1+\xi_{22}e_2)-e_2(\xi_{11}e_1+\xi_{21}e_2)\\
			&=\xi_{12}e_1^2+\xi_{21}e_2^2=\xi_{12}e_1=0\implies \xi_{12}=0;\\
			\theta(e_1,e_1)&=\varphi(e_1^2)-2e_1\varphi(e_1)=\xi_{11}e_1+\xi_{21}e_2-2\xi_{11}e_1=-\xi_{11}e_1+\xi_{21}e_2;\\
			\theta(e_2,e_2)&=\varphi(e_2^2)-2e_2\varphi(e_2)=0.
		\end{align*} 
		Hence, we have that 
		\[\mathcal{B}^2(\E_1)=\left\{\big(\begin{smallmatrix}
			-\xi_{11} & \xi_{21} \\ 0 & 0
		\end{smallmatrix}\big)\mid \xi_{11},\xi_{21}\in\mathbb{C}\right\}=\spa\left\{\big(\begin{smallmatrix}
			1 & 0 \\ 0 & 0
		\end{smallmatrix}\big),\big(\begin{smallmatrix}
			0 & 1 \\ 0 & 0
		\end{smallmatrix}\big)\right\},\]
		and
		\[\mathcal{H}^2(\E_1)=\spa\left\{\big(\begin{smallmatrix}
			0 & 0 \\ 1 & 0
		\end{smallmatrix}\big)+\mathcal{B}^2(\E_1),\big(\begin{smallmatrix}
			0 & 0 \\ 0 & 1
		\end{smallmatrix}\big)+\mathcal{B}^2(\E_1)\right\}.\]
		Consequently, the set of all infinitesimal formal evolution deformations of $\E_1$ (up to equivalence) is given by 
		\begin{align*}
			\operatorname{InfDef}(\E_1)&=\left\{\big(\begin{smallmatrix}
				1 & 0 \\ 0 & 0
			\end{smallmatrix}\big)+\left(\alpha\big(\begin{smallmatrix}
				0 & 0 \\ 1 & 0
			\end{smallmatrix}\big)+\beta\big(\begin{smallmatrix}
				0 & 0 \\ 0 & 1
			\end{smallmatrix}\big)\right)t\mid \alpha,\beta\in\mathbb{C}\right\}
			=\left\{\big(\begin{smallmatrix}
				1 & 0 \\ 0 & 0
			\end{smallmatrix}\big)+\big(\begin{smallmatrix}
				0 & 0 \\ \alpha & \beta
			\end{smallmatrix}\big)t\mid \alpha,\beta\in\mathbb{C}\right\}.
		\end{align*}
	\end{proof}
	\section{Formal degenerations of evolution algebras}\label{sec:5}
	Since the defining feature of evolution algebras is the existence of a natural basis,  the families of invertible linear maps $\{g_t\}_{t \neq 0}$ considered in the study of formal degenerations~\eqref{eq:action} in the context of evolution algebras will be required to map natural bases to natural bases.  This observation motivates the following definition.
	\begin{definition}
		Let $\mathcal{E} = (V, \mu)$ be an evolution algebra with natural basis $B = \{e_1,\dots,e_n\}$. A nonsingular matrix $g$ defines a \textit{natural basis change} if it represents the change of basis from $B$ to another natural basis $B' = \{f_1, \dots, f_n\}$. In this case, the product $\mu'$ of $\mathcal{E}$ with respect to $B'$ is given by
		\[\mu'(f_i,f_j)=(g\cdot\mu)(f_i,f_j)=g\big(\mu(g^{-1}f_i,g^{-1}f_j)\big),\]
		for all $i,j=1,\dots,n$, so in particular $\mu'(f_i,f_j)=0$ for all $i\neq j$.
	\end{definition}
	\begin{remark}
		The relation between two structure matrices of the same evolution algebra relative to different natural bases, together with the corresponding change of basis matrices, is described in \cite[Section 3.2.2]{Tian_08} and in more detail in \cite[Section~1.3]{thesis_yolanda}. 
		Let $\E$ be an evolution algebra with natural basis $B=\{e_1,\dots,e_n\}$. If $B'=\{f_1,\dots,f_n\}$ is another natural basis, let $g=P_{BB'}=(p_{ij})$ and $g^{-1}=P_{B'B}=(q_{ij})$ be the change of basis matrices, then
		\[M_{B'}(\E)^t=g\,M_B(\E)^t\,(g^{-1})^{(2)},\]
		where $(g^{-1})^{(2)}=(q_{ij}^2)$.
	\end{remark}
	We now introduce the notion of formal degenerations in the setting of evolution algebras. The following definition is inspired by the concept of a \textit{contraction} (see, for instance, references \cite{B_07_cont,FM_06_cont}), a special case of degeneration in the classical context of varieties.
	\begin{definition}
		Let $\mu$ and $\lambda$ be two evolution algebras. We say that $\lambda$ is a \textit{formal evolution degeneration} of $\mu$, or that $\mu$ \textit{formally degenerates} to $\lambda$, if there exists a continuous map $g\colon(0,1]\longrightarrow\operatorname{GL}(n,\mathbb{K})$, $t\mapsto g_t$, with each linear isomorphism $g_t$ defining a natural basis change of $\mu$, such that 
		\[\lambda=\lim_{t \to 0}g_t\cdot\mu.\]
		We denote this by $\mu\rightarrow\lambda$.
	\end{definition}
	\begin{remark}~
		\label{rem:(2)}
		\begin{enumerate}
			\item Since each $g_t$ is an isomorphism for $t\in(0,1]$, all the algebras $g_t\cdot\mu$ are isomorphic to $\mu$. Hence, to obtain a new evolution algebra via formal evolution degeneration one needs $\det{(g_0)}=0$. This is a necessary condition but not a sufficient one.
			\item Throughout this paper, we only consider matrices $g_t$ with entries of the form $\alpha t^m$, with $\alpha\in\mathbb{K}$, $m\in\mathbb{Z}$. Then, the structure constants of $g_t\cdot\mu$ will be elements of $\mathbb{K}[t,t^{-1}]$. However, if a degeneration exists, the structure constants necessarily lie on $\mathbb{K}[t]$ to ensure that the limit exists when $t\to0$.
			\item In what follows, we slightly abuse notation and denote by $B'=\{f_1, \dots, f_n\}$ both the natural basis of $g_t\cdot\mu$ induced by $g_t$ and that of the limiting algebra $\lambda$.
			\item As in the case of deformations, we will use the term degeneration to mean a formal evolution degeneration.
		\end{enumerate}
	\end{remark}
	\begin{example}\label{ex:abelian}
		Any evolution algebra $\mu$ degenerates to the abelian evolution algebra of the same dimension. Simply consider the matrix $g_t=t^{-1}I_n$. Then, we have
		\[g_t\cdot\mu(f_i,f_i)=g_t\big(\mu(g_t^{-1}f_i,g_t^{-1}f_i)\big)=t^{-1}\big(\mu(te_i,te_i)\big)=t\mu(e_i,e_i)\xrightarrow{\text{$t\to0$}}0.\]
	\end{example}
	
	The following  proposition gives a connection between degeneration and deformation.
	\begin{proposition}\label{prop:deg_imp_def}
		If $\lambda$ is a degeneration of $\mu$ (in the setting of the second item of Remark~\ref{rem:(2)}), then $\mu$ is a deformation of $\lambda$.
	\end{proposition}
	\begin{proof}
		By hypothesis, there exists a matrix $g_t$ which is nonsingular for all $t\in(0,1]$ and whose entries are elements of the form $\alpha t^m$ such that $\lambda=\lim_{t\to0}g_t\cdot\mu$. Then, the structure constants of $g_t\cdot\mu$ are elements of $\mathbb{K}[t]$. Since $g_t\cdot\mu\cong\mu$ for all $t\in(0,1]$, it follows that $\mu$ is a deformation of $\lambda$. 
	\end{proof}
	\subsection{Main properties and weaknesses}
	Motivated by the necessary conditions commonly used to establish the existence of a degeneration in the classical setting of Lie algebras (see \cite[Proposition~1.8]{B_05_7dim_nilp}), we state the following result.
	\begin{proposition}\label{prop:dim}
		Let $\mu$ and $\lambda$ be two $n$-dimensional evolution algebras such that $\lambda$ is a degeneration of $\mu$. Then, the following assertions hold:
		\begin{enumerate}[\rm(i)]
			\item $\dim{\ann(\mu)}\leq\dim{\ann(\lambda)}$;
			\item if there exists an integer $k>0$ such that $\dim{\ann^i(\mu)}=\dim{\ann^i(\lambda)}$ for all $i\leq k$, then $\dim{\ann^{k+1}(\mu)}\leq\dim{\ann^{k+1}(\lambda)}$, that is, the type of $\mu$ is less (with the lexicographic order) than the type of $\lambda$;
			\item $\dim{\lambda^2}\leq\dim{\mu^2}$.
			\item $\dim{\mathcal{B}^2(\mu)}\geq\dim{\mathcal{B}^2(\lambda)}$.
			\item $\dim{\mathcal{H}^2(\mu)}\leq\dim{\mathcal{H}^2(\lambda)}$.
		\end{enumerate}	
	\end{proposition}
	\begin{proof}
		By hypothesis, there exists a continuous map $g\colon(0,1]\longrightarrow\operatorname{GL}(n,\mathbb{K})$, $t\mapsto g_t$, with each linear isomorphism $g_t$ defining a natural basis change of $\mu$ such that $\lambda=\lim_{t\to0}g_t\cdot\mu$.
		\begin{enumerate}[\rm (i)]
			\item Assume that $\dim{\ann(\mu)}=r$. Then, since $g_t\cdot\mu\cong\mu$ for all $t\in(0,1]$, we have that $\dim{\ann(g_t\cdot\mu)}=r$. Say, without loss of generality, that $\ann(g_t\cdot\mu)=\spa\{f_1,\dots,f_r\}$. Since $g_t\cdot\mu(f_1,f_1)=\dots=g_t\cdot\mu(f_r,f_r)=0$, we also have that 
			\[\lambda(f_i,f_i)=\lim_{t\to0}g_t\cdot\mu(f_i,f_i)=0,\]
			for all $1\leq i\leq r$. Consequently, $\spa\{f_1,\dots,f_r\}\subset\ann(\lambda)$, and the inequality follows.
			\item Assume that $\dim{\ann^{k}(\mu)}=m$ and $\dim{\ann^{k+1}(\mu)}=r$, with $m\leq r$. Reasoning as before, we can say $\ann^{k}(g_t\cdot\mu)=\spa\{f_1,\dots,f_m\}$ and $\ann^{k+1}(g_t\cdot\mu)=\spa\{f_1,\dots,f_r\}$. Since $g_t\cdot\mu(f_i,f_i)\in\ann^{k}=\spa\{f_1,\dots,f_m\}$ for all $1\leq i\leq r$, we also have 
			\[\lambda(f_i,f_i)=\lim_{t\to0}g_t\cdot\mu(f_i,f_i)\in\spa\{f_1,\dots,f_m\}=\ann^k(\lambda),\]
			for all $1\leq i\leq r$, where the last equality follows from the hypothesis that $\dim{\ann^i(\mu)}=\dim{\ann^i(\lambda)}$ for all $i\leq k$. Consequently, $\spa\{f_1,\dots,f_r\}\subset\ann^{k+1}(\lambda)$, and the inequality follows.
			\item Since $g_t\cdot\mu\cong\mu$ for all $t\in(0,1]$, we have
			\begin{align*}
				\dim{\mu^2}&=\dim{\big(\spa\{\mu(e_i,e_i)\colon i=1,\dots,n\}\big)}\\&=\dim{\big(\spa\{g_t\cdot\mu(f_i,f_i)\colon i=1,\dots,n\}\big)}=\dim{(g_t\cdot\mu)^2}.
			\end{align*}
			Moreover, it is easy to check that if $\sum_{i=1}^n\alpha_i\big(g_t\cdot\mu(e_i,e_i)\big)=0$ for some $\alpha_i\in\mathbb{K}$, then $\sum_{i=1}^n\alpha_i\lambda(f_i,f_i)=0$. Indeed,
			\begin{align*}
				\sum_{i=1}^n\alpha_i\lambda(f_i,f_i)
				& =\sum_{i=1}^n\alpha_i\lim_{t\to0}g_t\cdot\mu(f_i,f_i) =\lim_{t\to0}\sum_{i=1}^n\alpha_i\big(g_t\cdot\mu(f_i,f_i)\big)=0.
			\end{align*}
			Since the linear dependence relations are preserved, then we get $\dim{\mu^2}=\dim{(g_t\cdot\mu)^2}\geq\dim{\lambda^2}$.
			\item We show that every element in $\mathcal{B}^2(\lambda)$ can be obtained as the limit when $t\to0$ of an element in $\mathcal{B}^2(g_t\cdot\mu)$. Let $\theta\in\mathcal{B}^2(\lambda)$. Then, for all $i,j=1,\dots,n$ we have that
			\begin{align*}
				\theta(f_i,f_j)&=\varphi(\lambda(f_i,f_j))-\lambda(f_i,\varphi(f_j))-\lambda(\varphi(f_i),f_j)\\
				&\overset{(\ast)}{=}\lim_{t\to0}\varphi(g_t\cdot\mu(f_i,f_j))-\lim_{t\to0}g_t\cdot\mu(f_i,\varphi(f_j))-\lim_{t\to0}g_t\cdot\mu(\varphi(f_i),f_j)\\
				&=\lim_{t\to0}\Big[\varphi(g_t\cdot\mu(f_i,f_j))-g_t\cdot\mu(f_i,\varphi(f_j))-g_t\cdot\mu(\varphi(f_i),f_j)\Big]=\lim_{t\to0}\theta_t,
			\end{align*}
			for some $\theta_t\in\mathcal{B}^2(g_t\cdot\mu)$. The equality $(\ast)$ follows from the fact that the limit and the map $\varphi$ commute. Indeed, when $i\neq j$ it is trivial; and when $i=j$, if we denote by $\mu_{ik}(t)$ the structure constants of $g_t\cdot\mu$ and by $\lambda_{ik}$ the structure constants of $\lambda$, we have
			\begin{align*}
				\varphi(\lambda(f_i,f_i))&=\varphi\left(\sum_{k=1}^n\lambda_{ik}f_k\right)=\sum_{k=1}^n\lambda_{ik}\varphi(f_k)=\sum_{k=1}^n\big(\lim_{t\to0}\mu_{ik}(t)\big)\varphi(f_k)\\
				&=\sum_{k=1}^n\lim_{t\to0}\big(\mu_{ik}(t)\varphi(f_k)\big)=\lim_{t\to0}\sum_{k=1}^n\mu_{ik}(t)\varphi(f_k)\\
				&=\lim_{t\to0}\varphi\left(\sum_{k=1}^n\mu_{ik}(t)f_k\right)=\lim_{t\to0}\varphi\big(g_t\cdot\mu(f_i,f_i)\big).
			\end{align*}
			Finally, since $g_t\cdot\mu\cong\mu$ for all $t\in(0,1]$, we conclude that $\dim{\mathcal{B}^2(\mu)}=\dim{\mathcal{B}^2(g_t\cdot\mu)}\geq\dim{\mathcal{B}^2(\lambda)}$.
			\item This follows straightforwardly from the fact that for any given evolution algebra $\E$ it holds that $\dim{\mathcal{H}^2(\E)}=n^2-\dim{\mathcal{B}^2(\E)}$.
		\end{enumerate}	
	\end{proof}
	
	Unlike in the general setting of algebra varieties, degenerations in evolution algebras are not necessarily transitive. As we show in the next two examples, each step in a degeneration chain must be examined with particular care. In the first, we provide full computations to explicitly illustrate how such degenerations are constructed. These details will be mostly omitted in the rest of the paper.
	
	\begin{example}\label{ex:1}
		Consider the following two-dimensional evolution algebras:
		\[
		\mu_1: e_1^2 = e_1,\; e_2^2 = e_2; \qquad
		\mu_2: e_1^2 = e_1,\; e_2^2 = 0; \qquad
		\mu_3: e_1^2 = e_2,\; e_2^2 = 0.
		\]
		First, it is not difficult to see that
		\[
		\lim_{t \to 0} g_t \cdot \mu_1 =\mu_2, \quad \text{with } g_t =\begin{pmatrix} 1 & 0 \\ 0 & t^{-1} \end{pmatrix}\quad\left(\text{and } g_t^{-1} =\begin{pmatrix} 1 & 0 \\ 0 & t \end{pmatrix}\right).
		\]
		Explicitly, we have
		\begin{align*}
			g_t\cdot\mu_1(f_1,f_1)&=g_t\big(\mu_1(g_t^{-1}f_1,g_t^{-1}f_1)\big)=g_t\big(\mu_1(e_1,e_1)\big)=g_t(e_1)=f_1\xrightarrow{\text{$t\to0$}}f_1,\\
			g_t\cdot\mu_1(f_2,f_2)&=g_t\big(\mu_1(g_t^{-1}f_2,g_t^{-1}f_2)\big)=t^2g_t\big(\mu_1(e_2,e_2)\big)=t^2g_t(e_2)=tf_2\xrightarrow{\text{$t\to0$}}0;
		\end{align*}
		which converges to the product of $\mu_2$. In the same way, we can also see that
		\[
		\lim_{t \to 0} h_t \cdot \mu_2 \neq \mu_3, \quad \text{with } h_t = \begin{pmatrix} t^{-1} & 0 \\ t^{-2} & t^{-1} \end{pmatrix}\quad\left(\text{and } h_t^{-1} =\begin{pmatrix} t & 0 \\ -1 & t \end{pmatrix}\right).
		\]
		Explicitly, we have
		\begin{align*}
			h_t\cdot\mu_2(f_1,f_1)&=h_t\big(\mu_2(h_t^{-1}f_1,h_t^{-1}f_1)\big)=h_t\big(\mu_2(te_1-e_2,te_1-e_2)\big)\\&=h_t\big(t^2\mu_2(e_1,e_1)+\mu_2(e_2,e_2)\big)=t^2h_t(e_1)=tf_1+f_2\xrightarrow{\text{$t\to0$}}f_2,\\
			h_t\cdot\mu_2(f_2,f_2)&=h_t\big(\mu_2(g_t^{-1}f_2,g_t^{-1}f_2)\big)=t^2h_t\big(\mu_2(e_2,e_2)\big)=0\xrightarrow{\text{$t\to0$}}0;
		\end{align*}
		which converges to the product of $\mu_3$.
		However, there is no contraction from $\mu_1$ to $\mu_3$. Since $\mu_1$ is regular, it has a unique natural basis (see \cite[Theorem 4.4]{EL_15}). Hence, the only possibilities are given by
		\begin{align*}
			g_t&=\begin{pmatrix}
				g_1(t) & 0 \\ 0 & g_2(t)
			\end{pmatrix}\left(\text{and } g_t^{-1} =\begin{pmatrix} \big(g_1(t)\big)^{-1} & 0 \\ 0 & \big(g_2(t)\big)^{-1} \end{pmatrix}\right)\quad\text{or}\\
			h_t&=\begin{pmatrix}
				0 & h_1(t) \\ h_2(t) & 0
			\end{pmatrix}\left(\text{and } h_t^{-1} =\begin{pmatrix} 0 & \big(h_2(t)\big)^{-1} \\ \big(h_1(t)\big)^{-1} & 0 \end{pmatrix}\right),
		\end{align*}
		for some continuous maps $g_1(t)$, $g_1(t)$, $h_1(t)$ and $h_2(t)$ which are different from zero for any $t\in(0,1]$, but it is easy to check that $\lim_{t \to 0} g_t \cdot \mu_1 \neq \mu_3$ and $\lim_{t \to 0} h_t \cdot \mu_1 \neq \mu_3$. We check the first case, the other one is analogue. In fact,
		\begin{align*}
			g_t\cdot\mu_1(f_1,f_1)&=\big(g_1(t)\big)^{-2}g_t\big(\mu_1(e_1,e_1)\big)=\big(g_1(t)\big)^{-2}g_t(e_1)=\big(g_1(t)\big)^{-1}f_1,\\
			g_t\cdot\mu_1(f_2,f_2)&=\big(g_2(t)\big)^{-2}g_t\big(\mu_1(e_2,e_2)\big)=\big(g_2(t)\big)^{-2}g_t(e_2)=\big(g_2(t)\big)^{-1}f_2;
		\end{align*}
		which does not converge to $\mu_3$ when $t\to0$ for any continuous maps $g_1(t)$ and $g_2(t)$.
	\end{example}
	
	\begin{remark}\label{rem:with_example}
		One might think that the previous example fails because the matrix
		\[
		h_tg_t =\begin{pmatrix} t^{-1} & 0 \\ t^{-2} & t^{-2} \end{pmatrix}\quad\left(\text{and } (h_tg_t)^{-1}=g_t^{-1}h_t^{-1} =\begin{pmatrix} t & 0 \\ -t & t^2 \end{pmatrix}\right)
		\]
		does not define a natural basis change of $\mu_1$. However, in general, even when we have a chain of degenerations $\mu_1\rightarrow\mu_2\rightarrow\mu_3$ given by $g_t$ and $h_t$, respectively, such that $h_tg_t$ is a natural basis change of the first evolution algebra, we may still have
		\[\lim_{t\to0}(h_tg_t)\cdot\mu_1\neq\lim_{t\to0}h_t\cdot\Big(\lim_{t \to 0}g_t\cdot\mu_1\Big).\]
		For instance, consider the following  three-dimensional evolution algebras $\mu_{3,4}$, $\mu_{3,2}$ and $\mu_{3,1}$ in Table \ref{table:clas_nilp}. We can easily check that $\mu_{3,2}=\lim_{t\to0}g_t\cdot\mu_{3,4}$ with $g_t=\operatorname{diag}(1,t,t^2)$ and that $\mu_{3,1}=\lim_{t\to0}h_t\cdot\mu_{3,2}$ with $h_t=\operatorname{diag}(1,t,1)$. However, although $h_tg_t=\operatorname{diag}(1,t^2,t^2)$ is a natural basis change of $\mu_1$ for all $t\neq0$, we have that $\lim_{t\to0}(h_tg_t)\cdot\mu_1\neq\mu_3$. Specifically,
		\[(h_tg_t)\cdot\mu_1(f_2,f_2)=t^{-4}h_t(g_t(e_3))=t^{-2}e_2\xrightarrow{\text{$t\to0$}}\infty.\]
	\end{remark}

	As a consequence of the lack of transitivity discussed above, we introduce the following definition considering the transitive closure of the relation ``being a degeneration ($\rightarrow$)''.
	
	\begin{definition}
		Let $\mathcal{F}$ be a family of evolution algebras, and let $\mu,\lambda\in\mathcal{F}$. We say that $\lambda$ is a \textit{transitive degeneration} of $\mu$, or that $\mu$ \textit{transitively degenerates} to $\lambda$, among $\mathcal{F}$ if there exists a finite sequence of degenerations
		\[\mu=\mu_0\rightarrow\mu_1\rightarrow\dots\rightarrow\mu_k=\lambda,\] with each $\mu_i\in\mathcal{F}$. We denote this by $\mu\rightsquigarrow\lambda$.
	\end{definition}
	
	\subsection{Degenerations of nilpotent evolution algebras up to dimension four}
	The study of degeneration relations among nilpotent Lie algebras has been extensively developed (see, for instance, \cite{B_05_7dim_nilp,GO_88_deg_nilp_Lie,S_90_deg_nilp_Lie}). 
	Motivated by this, and taking into account that complex nilpotent evolution algebras up to dimension four have been completely classified (see Table~\ref{table:clas_nilp}), 
	we devote this final part to the study of transitive degenerations within these families of evolution algebras, explicitly constructing the corresponding maps $g_t$ and the associated Hasse diagrams, which collectively describe all such relations.
	
	\begin{notation}
		For convenience in what follows, we shall denote the evolution algebras listed in Table~\ref{table:clas_nilp} by 
		$\mu_{i,j}$ instead of $\E_{i,j}$. That is, each symbol $\mu_{i,j}$ will refer to the corresponding algebra $\E_{i,j}$ in the classification. Moreover, we denote by $E_{ij}$ the matrix having a $1$ in the $(i,j)$-entry and zeros elsewhere.
	\end{notation}

	\subsubsection{Hasse diagrams of nilpotent evolution algebras of dimensions two and three}
	
	The unique isomorphism classes in $\mathcal{N}_2(\mathbb{C})$ are $\mu_{2,1}:e_1^2=e_2^2=0$ and $\mu_{2,2}:e_1^2=e_2$, $e_2^2=0$. Then, since every evolution algebra degenerates to the abelian one of the same dimension (Example \ref{ex:abelian}), the corresponding Hasse diagram is $\mu_{2,2}\rightarrow\mu_{2,1}$.
	
	\begin{theorem}\label{th:deg_3}
		All transitive degenerations among $\mathcal{N}_3(\mathbb{C})$ are captured by the Hasse diagram $\mu_{3,4}\rightarrow\mu_{3,3}\rightarrow\mu_{3,2}\rightarrow\mu_{3,1}$. Moreover, $\mu_{3,2}$  is also a degeneration of $\mu_{3,4}$.
	\end{theorem}
	\begin{proof}
		From the type and the dimension of the derived subalgebra of each evolution algebra in $\mathcal{N}_3(\mathbb{C})$ (see Table~\ref{table:clas_nilp}) and Proposition~\ref{prop:dim}, it follows that the previous Hasse diagram realises the maximal number of transitive degenerations. Therefore, since every evolution algebra degenerates to the abelian one, it remains only to prove that $\mu_{3,4}\rightarrow\mu_{3,3}$, $\mu_{3,3}\rightarrow\mu_{3,2}$ and $\mu_{3,4 }\rightarrow\mu_{3,2}$. Although we omit the detailed computations, we explicitly list the suitable transformations in Table~\ref{tab:cont_3}.
		\begin{table}[H]
        \centering
			\setlength{\tabcolsep}{10pt}
			\renewcommand{\arraystretch}{1.3}
				\begin{tabular}{|| l | l  ||}
					\hline
					\textbf{Degeneration} &  $g_t$   \\\hline\hline
					$\mu_{3,4}\rightarrow\mu_{3,3}$ & $\operatorname{diag}(1,t,t^2)+E_{32}  $ \\
					$\mu_{3,3}\rightarrow\mu_{3,2}$ 	&  $\operatorname{diag}(t,1,t^2)$  \\
					$\mu_{3,4}\rightarrow\mu_{3,2}$  & $\operatorname{diag}(1,t,t^2)$ \\\hline
				\end{tabular}
			\caption{All degenerations in $\mathcal{N}_3(\mathbb{C})$ and the corresponding transformations.}
			\label{tab:cont_3}
		\end{table}
	\end{proof}
	Finally, we state the following straightforward results.
	\begin{corollary}
		The evolution algebra $\mu_{3,4}$ is the unique rigid evolution algebra in $\mathcal{N}_3(\mathbb{C})$.
	\end{corollary}
	\begin{corollary}
		Let $\mu,\lambda\in\mathcal{N}_3(\mathbb{C})$. Then, $\mu\rightsquigarrow\lambda$ if and only if $\mu\rightarrow\lambda$.
	\end{corollary}
	
	\subsubsection{Hasse diagrams of nilpotent evolution algebras of dimension four}
	We devote this final section to the study of (transitive) degenerations among evolution algebras of dimension four over $\mathbb{C}$. In particular, the next result presents all transitive degeneration relations that we have explicitly established within $\mathcal{N}_4(\mathbb{C})$. Although additional relations may exist beyond those shown, the following Hasse diagram offers a coherent and informative global picture of the known degeneration structure.
	\begin{proposition}\label{prop:deg_4}
		The Hasse diagram in Figure \ref{fig:deg_4} captures several transitive degeneration relations among $\mathcal{N}_4(\mathbb{C})$. Particularly, the diagram represents all transitive degenerations among the algebras $\{\mu_{4,i}\}_{i=1}^7$.
		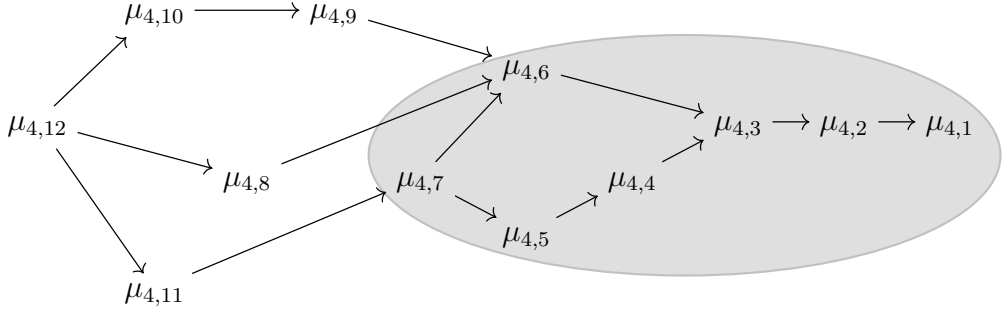
\begin{figure}[H]
			\centering
			\hspace{-13cm}\begin{tikzpicture}
				\draw[black, thick, fill=black!50, opacity=0.25]
				(9,0) ellipse (4.2cm and 1.6cm);  
				\begin{tikzcd}[cramped, sep=tiny]
					&  & {\mu_{4,10}} \arrow[rr]    &                            & {\mu_{4,9}} \arrow[rrrd] &                                      &  &                           &  &                         &  &                        &  &                        &  &             \\
					&  &                            &                            &                          &                                      &  & {\mu_{4,6}} \arrow[rrrrd] &  &                         &  &                        &  &                        &  &             \\
					{\mu_{4,12}} \arrow[rrddd] \arrow[rruu] \arrow[rrrd] &  &                            &                            &                          &                                      &  &                           &  &                         &  & {\mu_{4,3}} \arrow[rr] &  & {\mu_{4,2}} \arrow[rr] &  & {\mu_{4,1}} \\
					&  &                            & {\mu_{4,8}} \arrow[rrrruu] &                          & {\mu_{4,7}} \arrow[rruu] \arrow[rrd] &  &                           &  & {\mu_{4,4}} \arrow[rru] &  &                        &  &                        &  &             \\
					&  &                            &                            &                          &                                      &  & {\mu_{4,5}} \arrow[rru]   &  &                         &  &                        &  &                        &  &             \\
					&  & {\mu_{4,11}} \arrow[rrruu] &                            &                          &                                      &  &                           &  &                         &  &                        &  &                        &  &            
				\end{tikzcd}    
			\end{tikzpicture}
			\caption{Several transitive degeneration relations among $\mathcal{N}_4(\mathbb{C})$}
			\label{fig:deg_4}
		\end{figure}
		
	\end{proposition}
	\begin{proof}
		Although we omit the detailed computations, all these degenerations were explicitly constructed using suitable transformations $g_t$ listed in Table~\ref{tab:cont}. In particular, the grey ellipse highlights all transitive degenerations among the algebras $\mu_{4,1}$ through $\mu_{4,7}$. Looking at the type and the dimension of the derived subalgebras of $\mu_{4,4}$, $\mu_{4,5}$ and $\mu_{4,6}$, and applying Proposition~\ref{prop:dim}, it is immediate that neither $\mu_{4,4}$ nor $\mu_{4,5}$ degenerate to $\mu_{4,6}$, and conversely, $\mu_{4,6}$ does not degenerate to $\mu_{4,4}$ or $\mu_{4,5}$. Therefore, the Hasse diagram inside the grey ellipse above captures the maximal number of transitive degeneration relations within this family of seven evolution algebras, which, as shown in Table~\ref{tab:cont}, are all realisable.
		\begin{table}[H]
			\begin{center}
				\setlength{\tabcolsep}{10pt}
				\renewcommand{\arraystretch}{1.3}
				\begin{tabular}{|| l | l  || }
					\hline
					\textbf{Degeneration} & $g_t$  \\\hline\hline
					$\mu_{4,12}\longrightarrow\mu_{4,11}$ & $\operatorname{diag}(t,t^2,t^4,t^8)$ \\ 
					$\mu_{4,12}\longrightarrow\mu_{4,10}$ & $\operatorname{diag}(t^{-1},t^{-1},t^{-2},t^{-4})-t^{-2}E_{42}+t^{-2}E_{43}$ \\
					$\mu_{4,12}\longrightarrow\mu_{4,8}$ & $\operatorname{diag}(\sqrt{i}t^{-1},it^{-2},-t^{-2},t^{-4})+t^{-4}E_{42}-t^{-4}E_{43}$ \\
					$\mu_{4,11}\longrightarrow\mu_{4,7}$ & $\operatorname{diag}(1,1,t,t^{2})+E_{43}$ \\
					$\mu_{4,10}\longrightarrow\mu_{4,9}$ & $\operatorname{diag}(t,t,t^2,t^4)$  \\
					$\mu_{4,9}\longrightarrow\mu_{4,6}$ & $\operatorname{diag}(1,1,t,t^2)+E_{43}$ \\
					$\mu_{4,8}\longrightarrow\mu_{4,6}$ & $\operatorname{diag}(t,t^3,t^3,t^6)-it^2E_{43}$ \\
					$\mu_{4,7}\longrightarrow\mu_{4,6}$ & $\operatorname{diag}(1,t,t,t^2)+E_{42}$ \\
					$\mu_{4,7}\longrightarrow\mu_{4,5}$ & $\operatorname{diag}(t,t^2,1,t^4)$ \\
					$\mu_{4,6}\longrightarrow\mu_{4,3}$ & $\operatorname{diag}(t^2,t^2,1,t^4)$ \\
					$\mu_{4,5}\longrightarrow\mu_{4,4}$ & $\operatorname{diag}(1,t,1,t^2)+E_{32}$ \\
					$\mu_{4,4}\longrightarrow\mu_{4,3}$ & $\operatorname{diag}(t^{-3},t^{-1},t^{-4},t^{-2})+t^{-6}E_{43}$\\
					$\mu_{4,3}\longrightarrow\mu_{4,2}$ & $\operatorname{diag}(1,t^{-1},1,1)$\\
					\hline
				\end{tabular}
			\end{center}
			\caption{All degenerations in $\mathcal{N}_4(\mathbb{C})$ and the corresponding transformations..}
			\label{tab:cont}
		\end{table}
	\end{proof}
	\begin{remark}
		Although many of the transformations in Table~\ref{tab:cont} (and even earlier ones, such as those in Theorem~\ref{th:deg_3}) may appear difficult to obtain, Proposition~\ref{prop:deg_imp_def} often provides a powerful and practical tool. We illustrate its utility through the case $\mu_{4,12}\rightarrow\mu_{4,10}$. If $\mu_{4,10}$ is a degeneration of $\mu_{4,12}$, then, by Proposition~\ref{prop:deg_imp_def}, $\mu_{4,12}$ is a deformation of $\mu_{4,10}$. In particular, the structure matrix
		\[\begin{pmatrix}
			0 & 0 & 1 & 0 \\ 0 & 0 & 1 & 1 \\ 0 & 0 & 0 & 1 \\ 0 & 0 & 0 & 0
		\end{pmatrix}+t\begin{pmatrix}
			0 & 1 & 0 & 0 \\ 0 & 0 & 0 & 0 \\ 0 & 0 & 0 & 0 \\ 0 & 0 & 0 & 0
		\end{pmatrix}=\begin{pmatrix}
			0 & t & 1 & 0 \\ 0 & 0 & 1 & 1 \\ 0 & 0 & 0 & 1 \\ 0 & 0 & 0 & 0
		\end{pmatrix},\]
		defines a nilpotent first-order deformation of $\mu_{4,10}$, which is in fact isomorphic to $\mu_{4,12}$ for all $t\in\mathbb{C}^*$. A straightforward computation shows that the product with respect to the natural basis $B'=\{f_1=t^{-1}e_1,f_2=t^{-1}e_2-t^{-2}e_4,f_3=t^{-2}e_3+t^{-2}e_4,f_4=t^{-4}e_4\}$ coincides with the product of $\mu_{4,12}$. Indeed,
		\begin{align*}
			f_1^2&=t^{-2}e_1^2=t^{-2}(te_2+e_3)=t^{-1}e_2+t^{-2}e_3=f_2+f_3,\\
			f_2^2&=t^{-2}e_2^2+t^{-4}e_4^2=t^{-2}(e_2+e_3)=f_3,\\
			f_3^2&=t^{-4}e_3^2+t^{-4}e_4^2=t^{-4}e_4=f_4,\\
			f_4^2&=t^{-8}e_4^2=0.
		\end{align*}
		Therefore, a possible transformation $g_t$ is obtained as the change of basis matrix from $B'$ to the original basis $B$:
		\[g_t=\operatorname{diag}(t^{-1},t^{-1},t^{-2},t^{-4})-t^{-2}E_{42}+t^{-2}E_{43}.\]
	\end{remark}
	
	Although, as noted previously, the Hasse diagram of Proposition \ref{prop:deg_4} may not capture all possible transitive degenerations, it nevertheless allows us to draw the following conclusion.
	
	\begin{corollary}
		The evolution algebra $\mu_{4,12}$ is the unique rigid evolution algebra in $\mathcal{N}_4(\mathbb{C})$.
	\end{corollary}

	\section*{Acknowledgements}
	The second author was partially supported by the Agencia Estatal de Investigación (Spain), grant PID2020-115155GB-I00 (including European FEDER funding), by the Xunta de Galicia through the Competitive Reference Groups (GRC), grant ED431C 2023/31, and by the predoctoral contract FPU21/05685 and the research stay grant EST25/00293, both from the Ministerio de Ciencia, Innovación y Universidades (Spain).
	
	This paper was mostly written during a stay of Andrés Pérez-Rodríguez in Mulhouse (France). Andrés Pérez-Rodriguez is very grateful to the Département de Mathématiques, IRIMAS, Université de Haute-Alsace for their kind and hospitality.
	
	\section*{Declarations}
	
	\subsection*{Ethical Approval:}
	
	This declaration is not applicable.
	
	\subsection*{Conflicts of interest/Competing interests:} We have no conflicts of interests/competing interests to disclose.
	
	\subsection*{Authors' contributions:}
	
	All authors contributed equally to this work. 
	
	\subsection*{Data Availability Statement:} The authors confirm that the data supporting the findings of this study are available within the article.

	\newpage
	\appendix
	\section{Calculations for Theorem~\ref{th:cohom_2}}\label{appen:1}
	This first appendix is devoted to the explicit computations supporting Theorem~\ref{th:cohom_2}. These yield Table~\ref{table:inf_def}, which fully characterizes the infinitesimal deformations of all isomorphism classes of two-dimensional evolution algebras over $\mathbb{C}$.
	\subsection*{Infinitesimal deformations of \texorpdfstring{$\mathcal{E}_2$}{E2}}
	Let $\E_2$ be the evolution algebra over $\mathbb{C}$ or $\mathbb{R}$ with natural basis $B=\{e_1,e_2\}$ and product given by $e_1^2=e_2^2=e_1$. If $\theta\in\mathcal{B}^2(\E_2)$, then there must exist a linear morphism $(\varphi)_{BB}=(\xi_{ij})$ such that the following conditions hold:
	\begin{align*}
		\theta(e_1,e_2)&=\varphi(e_1e_2)-e_1\varphi(e_2)-\varphi(e_1)e_2=-(\xi_{12}+\xi_{21})e_1=0\implies \xi_{12}=-\xi_{21};\\
		\theta(e_1,e_1)&=\varphi(e_1^2)-2e_1\varphi(e_1)=-\xi_{11}e_1+\xi_{21}e_2;\\
		\theta(e_2,e_2)&=\varphi(e_2^2)-2e_2\varphi(e_2)=(\xi_{11}-2\xi_{22})e_1+\xi_{21}e_2.
	\end{align*} 
	Hence, we have that 
	\begin{align*}
		\mathcal{B}^2(\E_2)&=\left\{\big(\begin{smallmatrix}
			-\xi_{11} & \xi_{21} \\ \xi_{11}-2\xi_{22} & \xi_{21}
		\end{smallmatrix}\big)\mid \xi_{11},\xi_{21},\xi_{22}\in\mathbb{\mathbb{K}}\right\}=\spa\left\{\big(\begin{smallmatrix}
			-1 & 0 \\ 1 & 0
		\end{smallmatrix}\big),\big(\begin{smallmatrix}
			0 & 1 \\ 0 & 1
		\end{smallmatrix}\big),\big(\begin{smallmatrix}
			0 & 0 \\ -2 & 0
		\end{smallmatrix}\big)\right\},
	\end{align*}
	and
	\[\mathcal{H}^2(\E_2)=\spa\left\{\big(\begin{smallmatrix}
		0 & 0 \\ 0 & 1
	\end{smallmatrix}\big)+\mathcal{B}^2(\E_2)\right\}.\]
	Consequently, the set of all infinitesimal formal evolution deformations of $\E_2$ (up to equivalence) is given by 
	\begin{align*}
		\operatorname{InfDef}(\E_2)
		=\left\{\big(\begin{smallmatrix}
			1 & 0 \\ 1 & 0
		\end{smallmatrix}\big)+\big(\begin{smallmatrix}
			0 & 0 \\ 0 & \alpha
		\end{smallmatrix}\big)t\mid \alpha\in\mathbb{C}\right\}.
	\end{align*}
	\subsection*{Infinitesimal deformations of \texorpdfstring{$\mathcal{E}_3$}{E3}}
	Let $\E_3$ be the evolution algebra over $\mathbb{C}$ or $\mathbb{R}$ with natural basis $B=\{e_1,e_2\}$ and product given by $e_1^2=-e_2^2=e_1+e_2$. If $\theta\in\mathcal{B}^2(\E_3)$, then there must exist a linear morphism $(\varphi)_{BB}=(\xi_{ij})$ such that the following conditions hold:
	\begin{align*}
		\theta(e_1,e_2)&=\varphi(e_1e_2)-e_1\varphi(e_2)-\varphi(e_1)e_2=-(\xi_{12}-\xi_{21})(e_1+e_2)=0\implies \xi_{12}=\xi_{21};\\
		\theta(e_1,e_1)&=\varphi(e_1^2)-2e_1\varphi(e_1)=(\xi_{12}-\xi_{11})e_1+(\xi_{12}+\xi_{22}-2\xi_{11})e_2;\\
		\theta(e_2,e_2)&=\varphi(e_2^2)-2e_2\varphi(e_2)=(2\xi_{22}-\xi_{11}-\xi_{12})e_1+(\xi_{22}-\xi_{12})e_2.
	\end{align*} 
	Hence, we have that 
	\begin{align*}
		\mathcal{B}^2(\E_3)&=\left\{\big(\begin{smallmatrix}
			\xi_{12}-\xi_{11} & \xi_{12}+\xi_{22}-2\xi_{11} \\ 2\xi_{22}-\xi_{11}-\xi_{12} & \xi_{22}-\xi_{12}
		\end{smallmatrix}\big)\mid \xi_{11},\xi_{21},\xi_{22}\in\mathbb{\mathbb{K}}\right\}\\&=\spa\left\{\big(\begin{smallmatrix}
			1 & 1 \\ -1 & -1
		\end{smallmatrix}\big),\big(\begin{smallmatrix}
			-1 & -2 \\ -1 & 0
		\end{smallmatrix}\big),\big(\begin{smallmatrix}
			0 & 1 \\ 2 & 1
		\end{smallmatrix}\big)\right\}=\spa\left\{\big(\begin{smallmatrix}
			1 & 1 \\ -1 & -1
		\end{smallmatrix}\big),\big(\begin{smallmatrix}
			0 & 1 \\ 2 & 1
		\end{smallmatrix}\big)\right\},
	\end{align*}
	and
	\[\mathcal{H}^2(\E_3)=\spa\left\{\big(\begin{smallmatrix}
		0 & 0 \\ 1 & 0
	\end{smallmatrix}\big)+\mathcal{B}^2(\E_3),\big(\begin{smallmatrix}
		0 & 0 \\ 0 & 1
	\end{smallmatrix}\big)+\mathcal{B}^2(\E_3)\right\}.\]
	Consequently, the set of all infinitesimal formal evolution deformations of $\E_3$ (up to equivalence) is given by 
	\begin{align*}
		\operatorname{InfDef}(\E_3)
		=\left\{\big(\begin{smallmatrix}
			1 & 1 \\ -1 & -1
		\end{smallmatrix}\big)+\big(\begin{smallmatrix}
			0 & 0 \\ \alpha & \beta
		\end{smallmatrix}\big)t\mid \alpha\in\mathbb{K}\right\}.
	\end{align*}
	
	\subsection*{Infinitesimal deformations of \texorpdfstring{$\mathcal{E}_4$}{E4}}
	
	Let $\E_4$ be the evolution algebra over $\mathbb{C}$ or $\mathbb{R}$ with natural basis $B=\{e_1,e_2\}$ and product given by $e_1^2=e_2$ and $e_2^2=0$. If $\theta\in\mathcal{B}^2(\E_4)$, then there must exist a linear morphism $(\varphi)_{BB}=(\xi_{ij})$ such that the following conditions hold:
	\begin{align*}
		\theta(e_1,e_2)&=\varphi(e_1e_2)-e_1\varphi(e_2)-\varphi(e_1)e_2=-\xi_{12}e_2=0\implies \xi_{12}=0;\\
		\theta(e_1,e_1)&=\varphi(e_1^2)-2e_1\varphi(e_1)=(\xi_{22}-2\xi_{11})e_2;\\
		\theta(e_2,e_2)&=\varphi(e_2^2)-2e_2\varphi(e_2)=0.
	\end{align*} 
	Hence, we have that 
	\begin{align*}
		\mathcal{B}^2(\E_4)&=\left\{\big(\begin{smallmatrix}
			0 & \xi_{22}-2\xi_{11} \\ 0 & 0
		\end{smallmatrix}\big)\mid \xi_{11},\xi_{22}\in\mathbb{\mathbb{K}}\right\}=\spa\left\{\big(\begin{smallmatrix}
			0 & 1 \\ 0 & 0
		\end{smallmatrix}\big)\right\},
	\end{align*}
	and
	\[\mathcal{H}^2(\E_4)=\spa\left\{\big(\begin{smallmatrix}
		1 & 0 \\ 0 & 0
	\end{smallmatrix}\big)+\mathcal{B}^2(\E_4),\big(\begin{smallmatrix}
		0 & 0 \\ 1 & 0
	\end{smallmatrix}\big)+\mathcal{B}^2(\E_4),\big(\begin{smallmatrix}
		0 & 0 \\ 0 & 1
	\end{smallmatrix}\big)+\mathcal{B}^2(\E_4)\right\}.\]
	Consequently, the set of all infinitesimal formal evolution deformations of $\E_4$ (up to equivalence) is given by 
	\begin{align*}
		\operatorname{InfDef}(\E_4)
		=\left\{\big(\begin{smallmatrix}
			1 & 0 \\ 1 & 0
		\end{smallmatrix}\big)+\big(\begin{smallmatrix}
			\alpha & 0 \\ \beta & \gamma
		\end{smallmatrix}\big)t\mid \alpha,\beta,\gamma\in\mathbb{K}\right\}.
	\end{align*}
	
	\subsection*{Infinitesimal deformations of \texorpdfstring{$\mathcal{E}_5(a_2,a_3)$}{E5(a2,a3)}}
	
	Let $\E_5(a_2,a_3)$ be the evolution algebra over $\mathbb{C}$ or $\mathbb{R}$ with natural basis $B=\{e_1,e_2\}$ and product given by $e_1^2=e_1+a_2e_2$ and $e_2^2=a_3e_1+e_2$ with $1-a_2e_3\neq0$. If $\theta\in\mathcal{B}^2(\E_5(a_2,a_3))$, then there must exist a linear morphism $(\varphi)_{BB}=(\xi_{ij})$ such that the following conditions hold:
	\begin{align*}
		\theta(e_1,e_2)&=\varphi(e_1e_2)-e_1\varphi(e_2)-\varphi(e_1)e_2\\&=-\xi_{12}(e_1+a_2e_2)-\xi_{21}(a_3e_1+e_2)=0\implies \xi_{12}=\xi_{21}=0;\\
		\theta(e_1,e_1)&=\varphi(e_1^2)-2e_1\varphi(e_1)=-\xi_{11}e_1+a_2(\xi_{22}-2\xi_{11})e_2;\\
		\theta(e_2,e_2)&=\varphi(e_2^2)-2e_2\varphi(e_2)=a_3(\xi_{11}-2\xi_{22})e_1-\xi_{22}e_2.
	\end{align*} 
	Hence, we have that
	\begin{align*}
		\mathcal{B}^2(\E_5)&=\left\{\big(\begin{smallmatrix}
			-\xi_{11} & a_2(\xi_{22}-2\xi_{11}) \\ a_3(\xi_{11}-2\xi_{22}) & -\xi_{22}
		\end{smallmatrix}\big)\mid \xi_{11},\xi_{22}\in\mathbb{\mathbb{K}}\right\}=\spa\left\{\big(\begin{smallmatrix}
			-1 & -2a_2 \\ a_3 & 0
		\end{smallmatrix}\big),\big(\begin{smallmatrix}
			0 & a_2 \\ -2a_3 & -1
		\end{smallmatrix}\big)\right\}.
	\end{align*}
	Consequently, we now distinguish the set of infinitesimal deformations depending on the values of $a_2$ and $a_3$:
	\begin{enumerate}
		\item If $a_2\neq0$, then $\mathcal{H}^2(\E_5)=\spa\left\{\big(\begin{smallmatrix}
			0 & 0 \\ 1 & 0
		\end{smallmatrix}\big)+\mathcal{B}^2(\E_5),\big(\begin{smallmatrix}
			0 & 0 \\ 0 & 1
		\end{smallmatrix}\big)+\mathcal{B}^2(\E_5)\right\}$
		and
		\begin{align*}
			\operatorname{InfDef}(\E_5)
			=\left\{\big(\begin{smallmatrix}
				1 & a_2 \\ a_3 & 1
			\end{smallmatrix}\big)+\big(\begin{smallmatrix}
				0 & 0 \\ \alpha & \beta
			\end{smallmatrix}\big)t\mid \alpha,\beta\in\mathbb{K}\right\}.
		\end{align*}
		\item If $a_2=a_3=0$, then $\mathcal{H}^2(\E_5)=\spa\left\{\big(\begin{smallmatrix}
			0 & 1 \\ 0 & 0
		\end{smallmatrix}\big)+\mathcal{B}^2(\E_5),\big(\begin{smallmatrix}
			0 & 0 \\ 1 & 0
		\end{smallmatrix}\big)+\mathcal{B}^2(\E_5)\right\}$ and 
		\begin{align*}
			\operatorname{InfDef}(\E_5)
			=\left\{\big(\begin{smallmatrix}
				1 & 0 \\ 0 & 1
			\end{smallmatrix}\big)+\big(\begin{smallmatrix}
				0 & \alpha \\ \beta & 0
			\end{smallmatrix}\big)t\mid \alpha,\beta\in\mathbb{K}\right\}.
		\end{align*}
		\item If $a_2=0\neq a_3$, then $\mathcal{H}^2(\E_5)=\spa\left\{\big(\begin{smallmatrix}
			0 & 1 \\ 0 & 0
		\end{smallmatrix}\big)+\mathcal{B}^2(\E_5),\big(\begin{smallmatrix}
			0 & 0 \\ 0 & 1
		\end{smallmatrix}\big)+\mathcal{B}^2(\E_5)\right\}$ and 
		\begin{align*}
			\operatorname{InfDef}(\E_5)
			=\left\{\big(\begin{smallmatrix}
				1 & 0 \\ a_3 & 1
			\end{smallmatrix}\big)+\big(\begin{smallmatrix}
				0 & \alpha \\ 0 & \beta
			\end{smallmatrix}\big)t\mid \alpha,\beta\in\mathbb{K}\right\}.
		\end{align*}
	\end{enumerate}
	
	\subsection*{Infinitesimal deformations of \texorpdfstring{$\mathcal{E}_(a_4)$}{E6(a4)}}
	
	Let $\E_6(a_4)$ be the evolution algebra over $\mathbb{C}$ or $\mathbb{R}$ with natural basis $B=\{e_1,e_2\}$ and product given by $e_1^2=e_2$ and $e_2^2=e_1+a_4e_2$. If $\theta\in\mathcal{B}^2(\E_6)$, then there must exist a linear morphism $(\varphi)_{BB}=(\xi_{ij})$ such that the following conditions hold:
	\begin{align*}
		\theta(e_1,e_2)&=\varphi(e_1e_2)-e_1\varphi(e_2)-\varphi(e_1)e_2=-\xi_{21}e_1-(\xi_{12}+a_4\xi_{21})e_2=0\implies \xi_{12}=\xi_{21}=0;\\
		\theta(e_1,e_1)&=\varphi(e_1^2)-2e_1\varphi(e_1)=(\xi_{22}-2\xi_{11})e_2;\\
		\theta(e_2,e_2)&=\varphi(e_2^2)-2e_2\varphi(e_2)=(\xi_{11}-2\xi_{22})e_1-a_4\xi_{22}e_2.
	\end{align*} 
	Hence, we have that
	\begin{align*}
		\mathcal{B}^2(\E_6)&=\left\{\big(\begin{smallmatrix}
			0 & \xi_{22}-2\xi_{11} \\ \xi_{11}-2\xi_{22} & -a_4\xi_{22}
		\end{smallmatrix}\big)\mid \xi_{11},\xi_{22}\in\mathbb{\mathbb{K}}\right\}=\spa\left\{\big(\begin{smallmatrix}
			0 & 1 \\ -2 & -a_4
		\end{smallmatrix}\big),\big(\begin{smallmatrix}
			0 & -2 \\ 1 & 0
		\end{smallmatrix}\big)\right\}.
	\end{align*}
	and
	\[\mathcal{H}^2(\E_6)=\spa\left\{\big(\begin{smallmatrix}
		1 & 0 \\ 0 & 0
	\end{smallmatrix}\big)+\mathcal{B}^2(\E_6),\big(\begin{smallmatrix}
		0 & 0 \\ 0 & 1
	\end{smallmatrix}\big)+\mathcal{B}^2(\E_6)\right\}.\]
	Consequently, the set of all infinitesimal formal evolution deformations of $\E_4$ (up to equivalence) is given by 
	\begin{align*}
		\operatorname{InfDef}(\E_6)
		=\left\{\big(\begin{smallmatrix}
			0 & 1 \\ 1 & a_4
		\end{smallmatrix}\big)+\big(\begin{smallmatrix}
			\alpha & 0 \\ 0 & \beta
		\end{smallmatrix}\big)t\mid \alpha,\beta\in\mathbb{K}\right\}.
	\end{align*}
	\section{Calculations for Theorem \ref{th:deg_3}}\label{appen:2}
	
	This appendix is devoted to the explicit computations supporting Theorem~\ref{th:deg_3} and, particularly, Table~\ref{tab:cont_3}, which  fully characterises the transitive degeneration relations within $\mathcal{N}_3(\mathbb{C})$.
	
	\begin{table}[H]
		\renewcommand{\arraystretch}{1.5}
			\begin{tabular}{ ||p{3cm}|p{9.9cm}||  }
				\hline
				$\mu_{3,4}\longrightarrow\mu_{3,3}$ & $g_t=\operatorname{diag}(1,t,t^2)+E_{32}$\\
				\hline\hline
				\multicolumn{2}{||l||}{
					$g_t\cdot\mu_{3,4}(f_1,f_1)=g_t\big(\mu_{3,4}(e_1,e_1)\big)=g_t(e_2)=tf_2+f_3\xrightarrow{\text{$t\to0$}}f_3,$
				}\\
				\multicolumn{2}{||l||}{
					$g_t\cdot\mu_{3,4}(f_2,f_2)=t^4g_t\big(\mu_{3,4}(e_2,e_2)\big)=t^{-2}g_t(e_3)=f_3\xrightarrow{\text{$t\to0$}}f_3,$
				}\\
				\multicolumn{2}{||l||}{
					$g_t\cdot\mu_{3,4}(f_3,f_3)=t^{-4}g_t\big(\mu_{3,4}(e_3,e_3)\big)=0\xrightarrow{\text{$t\to0$}}0$
				}\\
				\hline
			\end{tabular}
	\end{table}
	\begin{table}[H]

		\renewcommand{\arraystretch}{1.5}
			\begin{tabular}{ ||p{3cm}|p{9.9cm}||  }
				\hline
				$\mu_{3,3}\longrightarrow\mu_{3,2}$ & $g_t=\operatorname{diag}(t,1,t^2)$\\
				\hline\hline
				\multicolumn{2}{||l||}{
					$g_t\cdot\mu_{3,3}(f_1,f_1)=t^{-2}g_t\big(\mu_{3,3}(e_1,e_1)\big)=t^{-2}g_t(e_3)=f_3\xrightarrow{\text{$t\to0$}}f_3;$
				}\\
				\multicolumn{2}{||l||}{
					$g_t\cdot\mu_{3,3}(f_2,f_2)=g_t\big(\mu_{3,3}(e_2,e_2)\big)=g_t(e_3)=t^2f_3\xrightarrow{\text{$t\to0$}}0;$
				}\\
				\multicolumn{2}{||l||}{
					$	g_t\cdot\mu_{3,3}(f_3,f_3)=t^{-4}g_t\big(\mu_{3,3}(e_3,e_3)\big)=0\xrightarrow{\text{$t\to0$}}0;$
				}\\
				\hline
			\end{tabular}
	\end{table}
	
	
	\begin{table}[H]
		\renewcommand{\arraystretch}{1.5}
			\begin{tabular}{ ||p{3cm}|p{9.9cm}||  }
				\hline
				$\mu_{3,4}\longrightarrow\mu_{3,2}$ & $g_t=\operatorname{diag}(1,t,t^2)$\\
				\hline\hline
				\multicolumn{2}{||l||}{
					$g_t\cdot\mu_{3,4}(f_1,f_1)=g_t\big(\mu_{3,4}(e_1,e_1)\big)=g_t(e_2)=tf_2\xrightarrow{\text{$t\to0$}}0;$
				}\\
				\multicolumn{2}{||l||}{
					$g_t\cdot\mu_{3,4}(f_2,f_2)=t^{-2}g_t\big(\mu_{3,4}(e_2,e_2)\big)=t^{-2}g_t(e_3)=f_3\xrightarrow{\text{$t\to0$}}f_3;$
				}\\
				\multicolumn{2}{||l||}{
					$	g_t\cdot\mu_{3,3}(f_3,f_3)=0\xrightarrow{\text{$t\to0$}}0;$
				}\\
				\hline
			\end{tabular}
	\end{table}

	\section{Calculations for Proposition \ref{prop:deg_4}}\label{appen:3}
	This appendix is devoted to the explicit computations supporting Proposition~\ref{prop:deg_4} and, particularly, Table~\ref{tab:cont}, which  yields several transitive degeneration relations within $\mathcal{N}_3(\mathbb{C})$.
	\begin{table}[H]
		\renewcommand{\arraystretch}{1.5}
			\begin{tabular}{ ||p{3cm}|p{11cm}||  }
				\hline
				$\mu_{4,12}\longrightarrow\mu_{4,11}$ & $g_t=\operatorname{diag}(t,t^2,t^4,t^8)$\\
				\hline\hline
				\multicolumn{2}{||l||}{
					$g_t\cdot\mu_{4,12}(f_1,f_1)=t^{-2}g_t\big(\mu_{4,12}(e_1,e_1)\big)=t^{-2}g_t(e_2+e_3)=f_2+t^2f_3\xrightarrow{\text{$t\to0$}}f_2;$
				}\\
				\multicolumn{2}{||l||}{
					$g_t\cdot\mu_{4,12}(f_2,f_2)=t^{-4}g_t\big(\mu_{4,12}(e_2,e_2)\big)=t^{-4}g_t(e_3)=f_3\xrightarrow{\text{$t\to0$}}f_3;$
				}\\
				\multicolumn{2}{||l||}{
					$g_t\cdot\mu_{4,12}(f_3,f_3)=t^{-8}g_t\big(\mu_{4,12}(e_3,e_3)\big)=t^{-8}g_t(e_4)=f_4\xrightarrow{\text{$t\to0$}}f_4;$
				}\\
				\multicolumn{2}{||l||}{
					$g_t\cdot\mu_{4,12}(f_4,f_4)=0\xrightarrow{\text{$t\to0$}}0.$
				}\\
				\hline
			\end{tabular}
	\end{table}
	
	\begin{table}[H]
		\renewcommand{\arraystretch}{1.4}
			\begin{tabular}{ ||p{3cm}|p{11cm}||  }
				\hline
				$\mu_{4,12}\longrightarrow\mu_{4,10}$ & $g_t = \operatorname{diag}(t^{-1},t^{-1},t^{-2},t^{-4})-t^{-2}E_{42}+t^{-2}E_{43}$\\
				\hline\hline
				\multicolumn{2}{||l||}{
					$g_t\cdot\mu_{4,12}(f_1,f_1)=t^{2}g_t\big(\mu_{4,12}(e_1,e_1)\big)=t^{2}g_t(e_2+e_3)=tf_2+f_3\xrightarrow{\text{$t\to0$}}f_3;$
				}\\
				\multicolumn{2}{||l||}{
					$g_t\cdot\mu_{4,12}(f_2,f_2)=t^{2}g_t\big(\mu_{4,12}(e_2,e_2)\big)=t^{2}g_t(e_3)=f_3+f_4\xrightarrow{\text{$t\to0$}}f_3+f_4;$
				}\\
				\multicolumn{2}{||l||}{
					$g_t\cdot\mu_{4,12}(f_3,f_3)=t^{4}g_t\big(\mu_{4,12}(e_3,e_3)\big)=t^{4}g_t(e_4)=f_4\xrightarrow{\text{$t\to0$}}f_4;$
				}\\
				\multicolumn{2}{||l||}{
					$g_t\cdot\mu_{4,12}(f_4,f_4)=0\xrightarrow{\text{$t\to0$}}0.$
				}\\
				\hline
			\end{tabular}
	\end{table}

	
	\begin{table}[H]
		\renewcommand{\arraystretch}{1.4}
			\begin{tabular}{ ||p{3cm}|p{11cm}||  }
				\hline
				$\mu_{4,12}\longrightarrow\mu_{4,8}$ & $g_t = \operatorname{diag}(\sqrt{i}t^{-1},it^{-2},-t^{-2},t^{-4})+t^{-4}E_{42}-t^{-4}E_{43}$\\
				\hline\hline
				\multicolumn{2}{||l||}{
					$g_t\cdot\mu_{4,12}(f_1,f_1)=\frac{t^2}{i}g_t\big(\mu_{4,12}(e_1,e_1)\big)=\frac{t^2}{i}g_t(e_2+e_3)=f_2+if_3\xrightarrow{\text{$t\to0$}}f_2+if_3;$
				}\\
				\multicolumn{2}{||l||}{
					$g_t\cdot\mu_{4,12}(f_2,f_2)=-t^{4}g_t\big(\mu_{4,12}(e_2,e_2)\big)=-t^{4}g_t(e_3)=t^2f_3+f_4\xrightarrow{\text{$t\to0$}}f_4;$
				}\\
				\multicolumn{2}{||l||}{
					$g_t\cdot\mu_{4,12}(f_3,f_3)=t^{4}g_t\big(\mu_{4,12}(e_3,e_3)\big)=t^{4}g_t(e_4)=f_4\xrightarrow{\text{$t\to0$}}f_4;$
				}\\
				\multicolumn{2}{||l||}{
					$g_t\cdot\mu_{4,12}(f_4,f_4)=0\xrightarrow{\text{$t\to0$}}0.$
				}\\
				\hline
			\end{tabular}
	\end{table}
	
	
	\begin{table}[H]
		\renewcommand{\arraystretch}{1.4}
			\begin{tabular}{ ||p{3cm}|p{11cm}||  }
				\hline
				$\mu_{4,11}\longrightarrow\mu_{4,7}$ & $g_t = \operatorname{diag}(1,1,t,t^{2})+E_{43}$\\
				\hline\hline
				\multicolumn{2}{||l||}{
					$g_t\cdot\mu_{4,11}(f_1,f_1)=g_t\big(\mu_{4,11}(e_1,e_1)\big)=g_t(e_2)=f_2\xrightarrow{\text{$t\to0$}}f_2;$
				}\\
				\multicolumn{2}{||l||}{
					$g_t\cdot\mu_{4,11}(f_2,f_2)=g_t\big(\mu_{4,11}(e_2,e_2)\big)=g_t(e_3)=tf_3+f_4\xrightarrow{\text{$t\to0$}}f_4;$
				}\\
				\multicolumn{2}{||l||}{
					$g_t\cdot\mu_{4,11}(f_3,f_3)=t^{-2}g_t\big(\mu_{4,11}(e_3,e_3)\big)=t^{4}g_t(e_4)=f_4\xrightarrow{\text{$t\to0$}}f_4;$
				}\\
				\multicolumn{2}{||l||}{
					$g_t\cdot\mu_{4,11}(f_4,f_4)=0\xrightarrow{\text{$t\to0$}}0.$
				}\\
				\hline
			\end{tabular}
	\end{table}
	
	
	\begin{table}[H]
		\renewcommand{\arraystretch}{1.4}
			\begin{tabular}{||p{3cm}|p{11cm}||}
				\hline
				$\mu_{4,10}\longrightarrow\mu_{4,9}$ & $g_t = \operatorname{diag}(t,t,t^2,t^4)$\\
				\hline\hline
				\multicolumn{2}{||l||}{
					$g_t\cdot\mu_{4,10}(f_1,f_1)=t^{-2}g_t\big(\mu_{4,10}(e_1,e_1)\big)=t^{-2}g_t(e_3)=f_3\xrightarrow{\text{$t\to0$}}f_3;$
				}\\
				\multicolumn{2}{||l||}{
					$g_t\cdot\mu_{4,10}(f_2,f_2)=t^{-2}g_t\big(\mu_{4,10}(e_2,e_2)\big)=t^{-2}g_t(e_3+e_4)=f_3+t^2f_4\xrightarrow{\text{$t\to0$}}f_3;$
				}\\
				\multicolumn{2}{||l||}{
					$g_t\cdot\mu_{4,10}(f_3,f_3)=t^{-4}g_t\big(\mu_{4,10}(e_3,e_3)\big)=t^{-4}g_t(e_4)=f_4\xrightarrow{\text{$t\to0$}}f_4;$
				}\\
				\multicolumn{2}{||l||}{
					$g_t\cdot\mu_{4,10}(f_4,f_4)=0\xrightarrow{\text{$t\to0$}}0.$
				}\\
				\hline
			\end{tabular}
	\end{table}
	
	
	\begin{table}[H]
		\renewcommand{\arraystretch}{1.4}
			\begin{tabular}{||p{3cm}|p{11cm}||}
				\hline
				$\mu_{4,9}\longrightarrow\mu_{4,6}$ & $g_t =\operatorname{diag}(1,1,t,t^2)+E_{43}$\\
				\hline\hline
				\multicolumn{2}{||l||}{
					$g_t\cdot\mu_{4,9}(f_1,f_1)=g_t\big(\mu_{4,9}(e_1,e_1)\big)=g_t(e_3)=tf_3+f_4\xrightarrow{\text{$t\to0$}}f_4;$
				}\\
				\multicolumn{2}{||l||}{
					$g_t\cdot\mu_{4,9}(f_2,f_2)=g_t\big(\mu_{4,9}(e_2,e_2)\big)=g_t(e_3)=tf_3+f_4\xrightarrow{\text{$t\to0$}}f_4;$
				}\\
				\multicolumn{2}{||l||}{
					$g_t\cdot\mu_{4,9}(f_3,f_3)=t^{-2}g_t\big(\mu_{4,9}(e_3,e_3)\big)=t^{-2}g_t(e_4)=f_4\xrightarrow{\text{$t\to0$}}f_4;$
				}\\
				\multicolumn{2}{||l||}{
					$g_t\cdot\mu_{4,9}(f_4,f_4)=0\xrightarrow{\text{$t\to0$}}0.$
				}\\
				\hline
			\end{tabular}
	\end{table}
	
	
	\begin{table}[H]
		\renewcommand{\arraystretch}{1.4}
			\begin{tabular}{ ||p{3cm}|p{11cm}||  }
				\hline
				$\mu_{4,8}\longrightarrow\mu_{4,6}$ & $g_t =\operatorname{diag}(t,t^3,t^3,t^6)-it^2E_{43}$\\
				\hline\hline
				\multicolumn{2}{||l||}{
					$g_t\cdot\mu_{4,8}(f_1,f_1)=t^{-2}g_t\big(\mu_{4,8}(e_1,e_1)\big)=t^{-2}g_t(e_2+ie_3)=tf_2+itf_3+f_4\xrightarrow{\text{$t\to0$}}f_4;$
				}\\
				\multicolumn{2}{||l||}{
					$g_t\cdot\mu_{4,8}(f_2,f_2)=t^{-6}g_t\big(\mu_{4,8}(e_2,e_2)\big)=t^{-6}g_t(e_4)=f_4\xrightarrow{\text{$t\to0$}}f_4;$
				}\\
				\multicolumn{2}{||l||}{
					$g_t\cdot\mu_{4,8}(f_3,f_3)=t^{-6}g_t\big(\mu_{4,8}(e_3,e_3)\big)=t^{-6}g_t(e_4)=f_4\xrightarrow{\text{$t\to0$}}f_4;$
				}\\
				\multicolumn{2}{||l||}{
					$g_t\cdot\mu_{4,8}(f_4,f_4)=0\xrightarrow{\text{$t\to0$}}0.$
				}\\
				\hline
			\end{tabular}
	\end{table}
	
	
	\begin{table}[H]
		\renewcommand{\arraystretch}{1.4}
			\begin{tabular}{ ||p{3cm}|p{11cm}||  }
				\hline
				$\mu_{4,7}\longrightarrow\mu_{4,6}$ & $ g_t =\operatorname{diag}(1,t,t,t^2)+E_{42}$\\
				\hline\hline
				\multicolumn{2}{||l||}{
					$g_t\cdot\mu_{4,7}(f_1,f_1)=g_t\big(\mu_{4,7}(e_1,e_1)\big)=g_t(e_2)=tf_2+f_4\xrightarrow{\text{$t\to0$}}f_4;$
				}\\
				\multicolumn{2}{||l||}{
					$g_t\cdot\mu_{4,7}(f_2,f_2)=t^{-4}g_t\big(\mu_{4,7}(e_2,e_2)\big)=t^{-4}g_t(e_4)=f_4\xrightarrow{\text{$t\to0$}}f_4;$
				}\\
				\multicolumn{2}{||l||}{
					$g_t\cdot\mu_{4,7}(f_3,f_3)=t^{-2}g_t\big(\mu_{4,7}(e_3,e_3)\big)=t^{-2}g_t(e_4)=f_4\xrightarrow{\text{$t\to0$}}f_4;$
				}\\
				\multicolumn{2}{||l||}{
					$g_t\cdot\mu_{4,7}(f_4,f_4)=g_t\big(\mu_{4,7}(e_4,e_4)\big)=0\xrightarrow{\text{$t\to0$}}0.$
				}\\
				\hline
			\end{tabular}
	\end{table}
	
	
	\begin{table}[H]
		\renewcommand{\arraystretch}{1.4}
			\begin{tabular}{ ||p{3cm}|p{11cm}||  }
				\hline
				$\mu_{4,7}\longrightarrow\mu_{4,5}$ & $ g_t = \operatorname{diag}(t,t^2,1,t^4)$\\
				\hline\hline
				\multicolumn{2}{||l||}{
					$g_t\cdot\mu_{4,7}(f_1,f_1)=t^{-2}g_t\big(\mu_{4,7}(e_1,e_1)\big)=t^{-2}g_t(e_2)=f_2\xrightarrow{\text{$t\to0$}}f_2;$
				}\\
				\multicolumn{2}{||l||}{
					$g_t\cdot\mu_{4,7}(f_2,f_2)=t^{-4}g_t\big(\mu_{4,7}(e_2,e_2)\big)=t^{-4}g_t(e_4)=f_4\xrightarrow{\text{$t\to0$}}f_4;$
				}\\
				\multicolumn{2}{||l||}{
					$g_t\cdot\mu_{4,7}(f_3,f_3)=g_t\big(\mu_{4,7}(e_3,e_3)\big)=g_t(e_4)=t^4f_4\xrightarrow{\text{$t\to0$}}0;$
				}\\
				\multicolumn{2}{||l||}{
					$g_t\cdot\mu_{4,7}(f_4,f_4)=g_t\big(\mu_{4,7}(e_4,e_4)\big)=0\xrightarrow{\text{$t\to0$}}0.$
				}\\
				\hline
			\end{tabular}
	\end{table}
	
	
	\begin{table}[H]
		\renewcommand{\arraystretch}{1.4}
			\begin{tabular}{ ||p{3cm}|p{11cm}||  }
				\hline
				$\mu_{4,6}\longrightarrow\mu_{4,3}$ & $ g_t = \operatorname{diag}(t^2,t^2,1,t^4)$\\
				\hline\hline
				\multicolumn{2}{||l||}{
					$g_t\cdot\mu_{4,6}(f_1,f_1)=t^{-4}g_t\big(\mu_{4,6}(e_1,e_1)\big)=t^{-4}g_t(e_4)=f_4\xrightarrow{\text{$t\to0$}}f_4$
				}\\
				\multicolumn{2}{||l||}{
					$g_t\cdot\mu_{4,6}(f_2,f_2)=t^{-4}g_t\big(\mu_{4,6}(e_2,e_2)\big)=t^{-4}g_t(e_4)=f_4\xrightarrow{\text{$t\to0$}}f_4;$
				}\\
				\multicolumn{2}{||l||}{
					$g_t\cdot\mu_{4,6}(f_3,f_3)=g_t\big(\mu_{4,6}(e_3,e_3)\big)=g_t(e_4)=t^4f_4\xrightarrow{\text{$t\to0$}}0;$
				}\\
				\multicolumn{2}{||l||}{
					$g_t\cdot\mu_{4,6}(f_4,f_4)=g_t\big(\mu_{4,6}(e_4,e_4)\big)=0\xrightarrow{\text{$t\to0$}}0.$
				}\\
				\hline
			\end{tabular}
	\end{table}
	
	
	\begin{table}[H]
		\renewcommand{\arraystretch}{1.4}
			\begin{tabular}{ ||p{3cm}|p{11cm}||  }
				\hline
				$\mu_{4,5}\longrightarrow\mu_{4,4}$ & $  g_t = \operatorname{diag}(1,t,1,t^2)+E_{32}$\\
				\hline\hline
				\multicolumn{2}{||l||}{
					$g_t\cdot\mu_{4,5}(f_1,f_1)=g_t\big(\mu_{4,5}(e_1,e_1)\big)=g_t(e_2)=tf_2+f_3\xrightarrow{\text{$t\to0$}}f_3;$
				}\\
				\multicolumn{2}{||l||}{
					$g_t\cdot\mu_{4,5}(f_2,f_2)=t^{-2}g_t\big(\mu_{4,5}(e_2,e_2)\big)=t^{-2}g_t(e_4)=f_4\xrightarrow{\text{$t\to0$}}f_4;$
				}\\
				\multicolumn{2}{||l||}{
					$g_t\cdot\mu_{4,5}(f_3,f_3)=g_t\cdot\mu_{4,5}(f_4,f_4)=0.$
				}\\
				\hline
			\end{tabular}
	\end{table}
	
	
	\begin{table}[H]
		\renewcommand{\arraystretch}{1.4}
			\begin{tabular}{ ||p{3cm}|p{11cm}||  }
				\hline
				$\mu_{4,4}\longrightarrow\mu_{4,3}$ & $ g_t = \operatorname{diag}(t^{-3},t^{-1},t^{-4},t^{-2})+t^{-6}E_{43}$\\
				\hline\hline
				\multicolumn{2}{||l||}{
					$g_t\cdot\mu_{4,4}(f_1,f_1)=t^{6}g_t\big(\mu_{4,4}(e_1,e_1)\big)=t^{6}g_t(e_3)=t^2f_3+f_4\xrightarrow{\text{$t\to0$}}f_4;$
				}\\
				\multicolumn{2}{||l||}{
					$g_t\cdot\mu_{4,4}(f_2,f_2)=t^{2}g_t\big(\mu_{4,4}(e_2,e_2)\big)=t^{2}g_t(e_4)=f_4\xrightarrow{\text{$t\to0$}}f_4;$
				}\\
				\multicolumn{2}{||l||}{
					$g_t\cdot\mu_{4,4}(f_3,f_3)=g_t\cdot\mu_{4,4}(f_4,f_4)=0.$
				}\\
				\hline
			\end{tabular}
	\end{table}
	
	
	\begin{table}[H]
		\renewcommand{\arraystretch}{1.4}
			\begin{tabular}{ ||p{3cm}|p{11cm}||  }
				\hline
				$\mu_{4,3}\longrightarrow\mu_{4,2}$ & $g_t = \operatorname{diag}(1,t^{-1},1,1)$\\
				\hline\hline
				\multicolumn{2}{||l||}{
					$g_t\cdot\mu_{4,3}(f_1,f_1)=g_t\big(\mu_{4,3}(e_1,e_1)\big)=g_t(e_3)=f_3\xrightarrow{\text{$t\to0$}}f_3;$
				}\\
				\multicolumn{2}{||l||}{
					$g_t\cdot\mu_{4,3}(f_2,f_2)=t^{2}g_t\big(\mu_{4,3}(e_2,e_2)\big)=t^{2}g_t(e_3)=t^2f_3\xrightarrow{\text{$t\to0$}}0;$
				}\\
				\multicolumn{2}{||l||}{
					$g_t\cdot\mu_{4,3}(f_3,f_3)=g_t\cdot\mu_{4,3}(f_4,f_4)=0.$
				}\\
				\hline
			\end{tabular}
	\end{table}

\end{document}